\documentclass[10pt,reqno]{amsart}

\usepackage[toc,page]{appendix}

\usepackage{xcolor}

\allowdisplaybreaks

\usepackage{boxedminipage}

\usepackage{cases}
\usepackage{placeins}
\usepackage{amsfonts,amsmath,amsthm}
\usepackage{amssymb,epsfig}
\usepackage{enumerate} 
\usepackage{fullpage}
\usepackage[utf8]{inputenc}
\usepackage{mathpazo}
\usepackage{mathtools}

\usepackage{eucal}
\usepackage[utf8]{inputenc}

\usepackage{todonotes}
\newcounter{todocounter}

\usepackage[unicode=true]{hyperref}
\hypersetup{colorlinks = true}
\hypersetup{
     colorlinks,
     linkcolor={black!10!red},
     linkbordercolor = {black!100!red},
     citecolor={blue}
}
\usepackage{pdfsync}

\theoremstyle{plain}
\newtheorem{theorem}{Theorem}[section]
\newtheorem{definition}[theorem]{Definition}
\newtheorem{question}{Question}

\newtheorem{lemma}[theorem]{Lemma}
\newtheorem{corollary}[theorem]{Corollary}
\newtheorem{proposition}[theorem]{Proposition}
\theoremstyle{remark}
\newtheorem{remark}{\bf{Remark}}
\numberwithin{equation}{section}

\newcommand{\R}{\mathbb{R}}
\newcommand{\T}{\mathbb{T}}
\newcommand{\x}{\mathbf{x}}

\newcommand{\F}{\mathbf{F}}
\newcommand{\K}{\mathbf{K}}
\newcommand{\C}{\mathrm{C}}

\usepackage{graphicx}

\usepackage{enumitem}
\makeatletter
\def\namedlabel#1#2{\begingroup
    #2%
    \def\@currentlabel{#2}%
    \phantomsection\label{#1}\endgroup
}
\makeatother

\usepackage{etoolbox}
\let\svlabel\label
\let\svref\eqref
\def\unusedlabels{}
\renewcommand\label[1]{\svlabel{#1}\global\edef\unusedlabels{\unusedlabels$<$#1$>$ }}
\renewcommand\ref[1]{\svref{#1}%
  \edef\teststring{$<$#1$>$}%
  \edef\tmp{\unusedlabels}%
  \def\unusedlabels{}%
  \expandafter\refhelper\tmp\relax%
}
\def\refhelper#1 #2\relax{%
  \edef\expandcase{#1}%
  \ifdefstrequal{\teststring}{\expandcase}{}{\edef\unusedlabels{\unusedlabels#1 }}%
  \if\relax#2\relax\else\refhelper#2\relax\fi
}

\title{\textsc{Polynomial Convergence Rate for Quasi-Periodic Homogenization of Hamilton-Jacobi Equations and Application to Ergodic Estimates}}
\author{Bingyang Hu, Son N.T. Tu$^\dagger$ and Jianlu Zhang}

\address{B. Hu \\ Department of Mathematics and Statistics, Auburn University}
\email{bzh0108@auburn.edu}

\address{Son. N. T. Tu$^\dagger$,\;Department of Mathematics, Michigan State University\\
East Lansing, Michigan 48824, USA}
    \email{tuson@msu.edu}

\address{Jianlu Zhang\\ Hua Loo-Keng Key Laboratory of Mathematics \& Mathematics Institute
            \\Academy of Mathematics and systems science
            \\Chinese Academy of Sciences, Beijing 100190, China}
        \email{jzhang87@amss.ac.cn}

\keywords{Quasi-periodic function; homogenization; rate of convergence;  Hamilton-Jacobi equations; viscosity solutions}
\subjclass[2010]{
35B27 
35B40 
35F21 
49L25 
}
\thanks{$\dagger$Corresponding author. \\
{\it Statements and Data Availability:} The authors declare no conflict of interest and no data associated with the paper. \\ The work of B. Hu is supported by SIMONS Travel Support MPS-TSM-00007213. 
S. Tu's work, initiated at UW-Madison, was partially supported by NSF grants DMS-1664424 and DMS-1843320 and, during his time at Michigan State University, by NSF grant DMS-2204722. The work of J. Zhang is supported by the National Key
R\&D Program of China (No. 2022YFA1007500) and the National Natural Science Foundation of China (No. 12231010).}
\date{\today}
\begin{document}

\begin{abstract} 
    In this paper, we demonstrate a polynomial convergence rate for homogenization of Hamilton–Jacobi equations with quasi-periodic potentials. We establish a connection between the convergence rate of homogenization and the regularity of the effective Hamiltonian, by 
    using a new quantitative ergodic estimate for \emph{bounded} quasi-periodic functions with Diophantine frequencies. As an application, we also study the convergent rate for Birkhoff average of  \emph{unbounded} quasi-periodic functions.
\end{abstract}

\maketitle



\section{Introduction}

\subsection{Motivation}  
For each $\varepsilon>0$, let $u^\varepsilon\in C\big(\mathbb{R}\times [0,\infty)\big)$ be the viscosity solution to:
\begin{equation}\label{eq:intro:Ceps}
\begin{cases}
\begin{aligned}
    u_t^\varepsilon + H\left(\frac{x}{\varepsilon},Du^\varepsilon\right) &= 0 & &\quad\text{in}\;\mathbb{R} \times (0,\infty),\\
    u^\varepsilon(x,0) &= u_0(x) & &\quad\text{on}\;\mathbb{R}.
\end{aligned}
\end{cases} \tag{$C_\varepsilon$}
\end{equation}
The initial data $u_0\in \mathrm{BUC}(\mathbb{R}) \cap \mathrm{Lip}(\mathbb{R})$, where $\mathrm{BUC}(\mathbb{R})$ denotes the set of bounded, uniformly continuous functions on $\mathbb{R}$, and the Hamiltonian $H=H(x,p):\R \times \R \to \R$ is a given continuous function satisfying:
\begin{itemize}
    \item[(i)] almost-periodic in $x$: for each $R>0$, the family of translation $\{ H(\cdot + z, \cdot ) : z \in  \R\}$ is relatively compact in $\mathrm{BUC}(\mathbb{R}^n\times B(0,R))$; 
    \item[(ii)] coercive in $p$ uniformly in $x$, that is $H(x,p)\to \infty$ as $|p|\to \infty$ uniformly in $x$.
\end{itemize}
As we can see, (i) naturally extends the periodic setting  (i.e., $H(x+\ell, p) = H(x, p)$ for $\ell\in \R$). In the almost-periodic setting (i)-(ii), it has been proven in \cite{Ishii_almost_periodic_1999, tran_hamilton-jacobi_2021} that $u^\varepsilon$ converges locally uniformly on $\mathbb{R} \times [0, \infty)$ to a function $u$ as $\varepsilon \to 0^+$, which solves the effective equation:
\begin{equation}\label{eq:intro:C}
\begin{cases}
\begin{aligned}
    u_t+ \overline{H}(Du) &= 0 & &\quad\text{in}\;\mathbb{R}\times(0,\infty),\\
    u(x,0) &= u_0(x) & &\quad\text{on}\;\mathbb{R}.
\end{aligned}
\end{cases} \tag{C}
\end{equation}
Here, the \emph{effective Hamiltonian} $\overline{H}(p)$ is determined as the unique constant such that the following \emph{approximated cell problem} \eqref{eq:CPdelta} can be solved with a viscosity solution $w_\delta\in \mathrm{BUC}(\mathbb{R})\cap \mathrm{Lip}(\mathbb{R})$ for any $\delta>0$:
\begin{equation}\label{eq:CPdelta} 
    \overline{H}(p) - \delta \leq H\big(x,p+Dw_\delta(x)\big) \leq \overline{H}(p)+\delta \qquad\text{in}\;\mathbb{R}. \tag{CP$_\delta$}
\end{equation}
Such a $w_\delta$ is also called an \emph{approximated corrector}.  For \(\delta = 0\), an exact corrector to \eqref{eq:CPdelta} may not exist \cite{lions_correctors_2003}.\color{black}

In this article, we are interested in the convergence rate of $u^\varepsilon-u$ as $\varepsilon\rightarrow 0^+$ for a class of mechanical Hamilton--Jacobi equations with almost-periodic potentials. 
A rate, expressed through an abstract modulus, was obtained in \cite{armstrong_error_2014, caffarelli_rates_2010} (see the literature section below). To the best of our knowledge, there has not been any result on an explicit polynomial rate of convergence for the almost periodic setting. Given the lack of compactness in this scenario, methods in the periodic setting cannot be applied directly.  
We adopt an approach similar to \cite{mitake_tran_yu_2019_rate_of_convergence_ARMA, tu_2018_rate_asymptotic}. Our key innovations include obtaining new quantitative ergodic estimates, integrating them into the framework, and establishing a connection between the regularity of \(\overline{H}\) and the rate of convergence.

\subsection{Literature and Contributions}

In the periodic setting, the homogenization for Hamilton-Jacobi equations was first initiated in \cite{LPV}. The perturbed test function method was introduced in \cite{evans_perturbed_1989, evans_homogenization_1992}, then combined with the discount approximation to achieve the first convergence rate of $\mathcal{O}(\varepsilon^{1/3})$ in \cite{capuzzo-dolcetta_rate_2001}.

In the convex, periodic setting, the optimal rate of convergence $\mathcal{O}(\varepsilon)$ under several conditions by using optimal control theory and weak KAM method was first established in \cite{mitake_tran_yu_2019_rate_of_convergence_ARMA}. A rate of $\mathcal{O}(\varepsilon\log(\varepsilon))$ was achieved in \cite{Cooperman_2021_ARMA} using a method rooted in percolation theory. Recently, the optimal rate $\mathcal{O}(\varepsilon)$ was attained in \cite{tran_optimal_2021} using a curve cutting method introduced in \cite{Burago1992}. Further developments of the method in \cite{tran_optimal_2021} can be found in  \cite{hanjang2023Nonlinearity, nguyen-tien_optimal_2022, han_quantitative_2024}. Earlier progress in this direction with nearly optimal rates can be found in \cite{jing_tran_yu_2019_minimax,tu_2018_rate_asymptotic}. See \cite{mitake_rate_2023} for a convergence rate in time-fractional Hamilton-Jacobi equations. We refer to \cite{shen_convergence_2015} for a convergence rate in the almost periodic elliptic setting. To date, optimal rates of convergence for general nonconvex first-order cases have not been established.

{Going beyond the periodic setting, the qualitative results (convergence of \(u^\varepsilon \to u\)) were established for non-convex Hamiltonians in almost periodic environments \cite{Ishii_almost_periodic_1999}, and for convex Hamilton-Jacobi equations in stationary ergodic environments in \cite{souganidis_stochastic_1999} and \cite{rezakhanlou_homogenization_2000}. In strongly mixing media, a logarithmic-type rate of convergence for the homogenization of fully nonlinear uniformly elliptic second-order equations was established in \cite{caffarelli_rates_2010}. When the Hamiltonian exhibits finite range dependence (a continuum analogue of i.i.d.) in its spatial dependence, \cite{armstrong_error_2014} provided an algebraic rate for first-order time-dependent Hamilton--Jacobi equations. These assumptions are not applicable to the almost periodic setting, although they independently derive a modulus-based rate for the almost periodic case using their new idea and the approach in \cite{capuzzo-dolcetta_rate_2001}.

In this paper, we adopt an optimal control-based framework, following \cite{mitake_tran_yu_2019_rate_of_convergence_ARMA, tu_2018_rate_asymptotic}, instead of the method in \cite{capuzzo-dolcetta_rate_2001}. On top of that, our primary contribution is the development of new quantitative ergodic estimates for functions in the quasi-periodic setting, which are essential for this framework. In the periodic case, \(\overline{H}\) is known to be \(C^1\) when \(n = 1\) and \(H(x,p) = \frac{1}{2}|p|^2 - V(x)\) with \(V \in C^2\) periodic \cite{barnard_nonlinearity_2005}. However, in the almost-periodic setting, the effective Hamiltonian \(\overline{H}\) may not be \(C^1\) (see Lemma \ref{lem:BehaviorDerevativeEffectiveHbar}).  Our second contribution establishes a connection between the regularity of \(\overline{H}\) (Lemmas \ref{lem:BehaviorDerevativeEffectiveHbar}, \ref{lem:BehaviorHolderDerevativeEffectiveHbar}) and the rate of convergence in homogenization. Additionally, as a by-product, we derive a convergence rate for the Birkhoff average of unbounded quasiperiodic functions, which is of independent interest (Proposition \ref{prop:ASingleCaseOfApplicationErgodic}). Another challenge in adapting the method from \cite{mitake_tran_yu_2019_rate_of_convergence_ARMA, tu_2018_rate_asymptotic} to the quasi-periodic setting is the potential nonexistence of an exact corrector for \(\delta = 0\) \cite{lions_correctors_2003}. We overcome this by identifying regions with exact correctors and integrating these with information of \(\overline{H}\) inside the optimal control formula (Proposition \ref{thm:lower-bound2D}).

\subsection{Setting and main results} Let $\mathbb{T}^n = \mathbb{R}^n\backslash \mathbb{Z}^n$ be the flat $n$-dimensional torus. We focus on the mechanical Hamiltonian $H:\mathbb{R}^2\to \mathbb{R}$ of the form
\begin{equation}\label{eq:assumptionH}
    H(x,p) = \frac{1}{2}|p|^2 + V(x), 
\end{equation}
with the assumption
\begin{description}[style=multiline, labelwidth=1cm, leftmargin=1.8cm] 
    \item[\namedlabel{itm:A0} {$(\mathcal{A}_0)$}] $V$ is {\it quasi-periodic}, i.e.,
    $V(x) = -\mathbf{U}(\xi x)$ for some $\mathbf{U}\in \mathrm{C}(\mathbb{T}^n)$, and {\it non-resonant frequency $\xi\in \mathbb{R}^n$ ($ \kappa\cdot \xi\neq 0$ for all $\kappa\in\mathbb Z^n\backslash\{0\}$}). Without loss of generality, we assume $\min_{\mathbb{T}^n} \mathbf{U} = 0$. 
\end{description}
Such a $\mathbf{U} \in \mathrm{C}(\mathbb{T}^n)$ is called a {\it suspension} of $V(x)$. As we can see, generic $\mathbf{U}$ has a unique minimizer on $\mathbb{T}^n$. For $n=2$, we can simplify the situation into two prototypes: 

\begin{description}[style=multiline, labelwidth=1cm, leftmargin=1.8cm] 
    \item[\namedlabel{itm:A1} {$(\mathcal{A}_1)$}] 
        $\mathbf{U}(\xi_1,\xi_2) = \left(2 - \sin(2\pi \xi_1) - \sin(2\pi \xi_2)\right)^\gamma$;
    \item[\namedlabel{itm:A2} {$(\mathcal{A}_2)$}]  $\mathbf{U}(\xi_1,\xi_2) = \left(2 - \cos(2\pi \xi_1) - \cos(2\pi \xi_2)\right)^\gamma$;
\end{description}
\color{black}
where $(\xi_1,\xi_2)\in \mathbb{T}^2$ and $ \gamma>0$. Under \ref{itm:A1}, $V(x)$ never achieves the maximum for $x\in\R$, while under \ref{itm:A2}, $V(x)$ associated achieves the maximum at a finite point $x=0$. However, there is no essential difference between \ref{itm:A1} and \ref{itm:A2} in estimating the convergence rate. By examining these prototypes (with adjustable $\gamma> 0$), we can handle general $V(x)$ similarly. We say $\xi\in \mathbb{R}^n$ has a Diophantine index $\sigma = \sigma_\xi > 0$ if there exists $C_\xi> 0$ such that $|\xi\cdot \kappa| \geq C_\xi |\kappa|^{-\sigma}$ for all $\kappa\in \mathbb{Z}^n\backslash\{0\}$.
It is known that if $\xi \in \R^n$ is non-resonant then $\sigma_\xi \geq n-1$ (see \cite{JurgenKAM1999}).

\begin{theorem}
\label{thm:rate-C1-H-bar-prototype} 
For $n=2$, we assume \ref{itm:A1} or \ref{itm:A2} and $\sigma_\xi = 1$. There exists $C = C(T)>0$ independent of $\varepsilon \in (0,1)$ such that, for $(x,t)\in \R\times [0,T]$ then
\begin{align}
    -C\varepsilon &\leq u^\varepsilon(x,t) - u(x,t) \leq 
    C\varepsilon^{\frac{\gamma-2}{3\gamma-2}} 
    & &  \text{if}\; \gamma >2, \label{eq:upper-bound-theta-gamma-bigger-2} \\
   -C\varepsilon^{1/2} &\leq u^\varepsilon(x,t) - u(x,t) \leq 
    \frac{C}{|\log(\varepsilon)|}, 
    & &   \text{if}\; \gamma =2.  \nonumber
\end{align}
Besides, we have the lower bound for all $\gamma>0$ as follows:
\begin{equation}\label{eq:lower-bound}
    u^\varepsilon(x,t) - u(x,t) \geq 
    \begin{cases}
        -C\varepsilon^{\frac{\gamma}{\gamma+1}}, 
                &\quad \gamma \in (0,1),\\
        -C\varepsilon^{1/2},       
                &\quad \gamma \in [1,2],
    \end{cases} \qquad\qquad\qquad\text{if}\; \gamma \in (0,2). 
\end{equation}
\end{theorem}
The smoothness of $\overline{H}$ (when $\gamma\geq 2$, for example, under condition \ref{itm:A1} or \ref{itm:A2}) is crucial for achieving an upper bound. Without this condition, only a lower bound is obtained as in \eqref{eq:lower-bound}.  In the quasi-periodic setting, \(\overline{H}\) may not be \(C^1\), as noted in Lemma \ref{lem:BehaviorDerevativeEffectiveHbar}, unlike in the periodic case where it is \(C^1\) (see \cite[Proposition 2.7]{barnard_nonlinearity_2005}). \color{black}

\begin{remark} \quad 
\begin{itemize}
\item[(i)] As an example, consider \( H(x,p) = \frac{1}{2}|p|^2 - \big(2 - \sin(2\pi x) - \sin(2\pi \sqrt{2}x)\big)^6 \) for \((x,p)\in \mathbb{R}^2\). Theorem \ref{thm:rate-C1-H-bar-prototype} implies that
\[
-C\varepsilon \leq u^\varepsilon(x,t) - u(x,t) \leq C\varepsilon^{1/4}.
\]
This example demonstrates that fractional convergence rates indeed occur (according to our scheme). However, it remains uncertain whether the \( \mathcal{O}(\varepsilon^{1/4}) \) rate is optimal for this example, and this requires further investigation. 
\item[(ii)] Inequality \eqref{eq:upper-bound-theta-gamma-bigger-2} actually holds for any frequency $\xi = (\xi_1,\xi_2)$ with $\sigma_\xi < 2$ (see Corollary \ref{coro:RelaxationOfSigmaXi2}). If $\omega \in\R$ is an \emph{algebraic irrational of degree} $d$, meaning it is a root of a polynomial of degree $d\in \mathbb{N}$ with integer coefficients, then $(\omega, 1)\in\R^2$ has a Diophantine index $ 1 + \delta$ with $\delta\in [0,1)$ (see  Roth's Theorem in \cite{cassels_1957}). In Theorem \ref{thm:rate-C1-H-bar-prototype}, condition $\sigma_\xi = 1$ holds if $\xi_1/\xi_2$ is algebraic irrational of degree $2$, for instance  $\xi = (1,\sqrt{2})$. 
\end{itemize}
\end{remark}
\color{black}

Inspired by Theorem \ref{thm:rate-C1-H-bar-prototype}, now we state our conclusion for general $n\geq 2$.  Let $\alpha\in (0,1]$\color{black}. The following conditions are necessary for different quantitative estimates (see Section  \ref{Sect:preliminaries-almost-periodic} for the definition of $H^s(\T^n)$ and  $\C^{0,\alpha}(\T^n)$):

\begin{description}[style=multiline, labelwidth=1cm, leftmargin=1.8cm] 
        \item[\namedlabel{itm:P1} {$(\mathcal{P}_1)$}] Assume $\xi \in \R^n$ has a Diophantine index $\sigma_\xi$, and $\mathbf{U}^{1/2}\in H^s(\T^n)$ for some $s> \frac{n}{2} + \sigma_\xi$.

        \item[\namedlabel{itm:P2} {$(\mathcal{P}_2)$}] Assume $\mathbf{U}^{1/2}\in \C^{0,\alpha}(\T^n)$ and $\xi \in \Omega_n \subset \R^n$, where $\Omega_n$ is a set of full measure (defined in \eqref{eq:defOmeganFullMeasure}).

        \item[\namedlabel{itm:P4New3} {$(\mathcal{P}_3)$}] Assume $\mathbf{U}^{1/2}\in \C^{0,\alpha}(\T^n)$ and $\xi = (\xi_1,\ldots, \xi_n)\in \R^n$ is non-resonant with $\xi_i$ is algebraic irrational for $i=1,\ldots, n$.
        
        \item[\namedlabel{itm:P3} {$(\mathcal{P}_4)$}] Assume $n=2$, $\mathbf{U}^{1/2}\in \C^{0,\alpha}(\T^2)$ and $\xi = (\xi_1,\xi_2) \in \R^2$ with $\sigma_\xi = 1$. 
    \end{description}

\begin{theorem}\label{thm:n-frequency-lower-bound} Assume \ref{itm:A0} and $n\geq 2$. There exists a constant $C = C(n, \alpha, \xi, \mathbf{U}^{1/2}, T)$ such that for all $(x,t)\in \R \times [0,T]$ then
\begin{numcases}
    {u^\varepsilon(x,t) - u(x,t) \geq}
        -C\varepsilon & \text{under \ref{itm:P1}}
        \label{eq:casesP1},\\
        -C \varepsilon^{\frac{\alpha}{\alpha+n-1}}|\log(\varepsilon)|^{3(n-1)}  & \text{under \ref{itm:P2}}
        \label{eq:casesP2},\\
        -C_\nu \varepsilon^{\frac{\alpha}{\alpha+n-1 + \nu} }& \text{under \ref{itm:P4New3}}
        \label{eq:casesP4New3},\\
         -C \varepsilon^{\frac{\alpha}{\alpha+1}} & \text{under \ref{itm:P3}}
         \label{eq:casesP3}. 
    \end{numcases}
    Furthermore, if the associated $\overline{H}\in \C^{1,\beta}(\R)$ then there exists $C = C(n, \alpha, \xi, \mathbf{U}^{1/2}, T)$ such that
\begin{numcases}
    {u^\varepsilon(x,t) - u(x,t) \leq}
        C\varepsilon^\frac{\beta}{\beta+1} & \text{under \ref{itm:P1}}
        \nonumber,\\
        C \varepsilon^{\frac{\beta}{\beta+1}\frac{\alpha}{\alpha+n-1}}|\log(\varepsilon)|^{3(n-1)}  & \text{under \ref{itm:P2}}
        \nonumber,\\
        C_\nu \varepsilon^{\frac{\beta}{\beta+1}\frac{\alpha}{\alpha+n-1 + \nu}} & \text{under \ref{itm:P4New3}}
        \label{eq:casesP4upper},\\
        C \varepsilon^{\frac{\beta}{\beta+1}\frac{\alpha}{\alpha+1}} & \text{under \ref{itm:P3}}
        \nonumber. 
\end{numcases}
In \eqref{eq:casesP4New3} or \eqref{eq:casesP4upper}, we mean that for every $0<\nu \ll 1$, there exists $C_\nu>0$ such that \eqref{eq:casesP4New3} or \eqref{eq:casesP4upper}, respectively, holds.
\end{theorem}

\begin{remark}  If $n=2$ then \ref{itm:P1} holds as long as $\mathbf{U}^{1/2}\in H^{s}(\T^2)$ for some $s>2$ and $\xi_1/\xi_2$ is algebraic irrational. 
We also point out, the validity of condition $\overline{H}\in \C^{1,\beta}$ relies on  specific forms of the potential $\mathbf{U}$, e.g. $\overline{H}\in \C^{1/2-1/\gamma}(\R)$ for \ref{itm:A1} or \ref{itm:A2} with $\gamma>2$. If $\overline{H}\in \C^1(\R)$ with $\overline{H}'$ continuous with respect to a modulus $\varpi$,  a rate of convergence can still be obtained in terms of $\varpi$. To avoid making the statement too technical, we do not include this point into Theorem \ref{thm:n-frequency-lower-bound}. 
\end{remark}

\subsection{Convergence rate of Birkhoff average under ergodic transformations of $\mathbb T^n$} During the proof of the upper bound estimate in Theorem \ref{thm:n-frequency-lower-bound}, we obtained the convergence rate of calibrated curves to their rotation vectors under the lower regularity of the effective Hamiltonian, which partially answers an open question raised in \cite[eq. (3) of Section 5.6]{tran_hamilton-jacobi_2021}. The question is about the rate of  convergence 
of general Birkhoff average under a torus translation with an irrational frequency. Although the {\it Birkhoff's ergodic theorem} guarantees the convergence under rather general conditions, quite few results have ever been made concerning the rate. For one dimensional torus maps, {\it Denjoy-Koksma’s inequality} proven by M. Herman in \cite[Chapter VI.3]{Herman1979French} gives the rate for functions with 
bounded variation, i.e. 
\begin{equation*}
    \left\Vert  \frac1 N\sum_{i=0}^{N-1} f(x+i\xi)-\int_{\mathbb T} fdx \right\Vert _{L^\infty}\leq \frac{{\rm Var}(f)}{N}
\end{equation*} 
for any $f\in BV(\mathbb T)$, $\xi\in\R\backslash\mathbb Q$ and $N\in\mathbb Z^+$ sufficiently large. Recently, this result has been generalized  to arbitrary dimension $n\geq 2$ in  \cite{KleinLiuMelo2021}, where a uniform rate of convergence was given for H\"older continuous functions $f\in \C^{0,\alpha}(\mathbb T^n)$ and Diophantine frequency $\xi\in\mathbb R^n$.

In this paper, we obtain a rate of Birkhoff average for \emph{unbounded functions} of a quasi-periodic setting. Precisely, for $f(x)={\bf V}(\xi x)$ satisfying \ref{itm:A0}, we proved the existence of the {\it mean value} of $1/f(x)$:
\begin{equation*}
    \mathcal{M}\left(\frac{1}{f}\right): =\lim_{T\to \infty} \frac{1}{T}\int_0^T \frac{dx}{f(x)},
\end{equation*}
and estimated the convergence rate of this limit. As we can see, $1/f(x)$ is of neither H\"older continuity nor bounded variation. So our result is totally new.

\begin{proposition}\label{prop:ASingleCaseOfApplicationErgodic} Assume \ref{itm:A1} and $\sigma_\xi = 1$. For $\gamma\in(0,\infty)$ we have: 
\begin{equation*}
\begin{aligned}
    \frac{1}{T}\int_0^T \frac{dx}{\mathbf{U}(\xi x)^{1/2}} &\geq  
    \begin{cases}
    \begin{aligned}
         & \textstyle CT^{\tau},  && \textstyle \tau = \frac{(\gamma-2)(3\gamma-2)}{(\gamma-2)(3\gamma-2)+4\gamma^2}, \quad & & \gamma \in (2,\infty), \\
        & \textstyle C|\log(T)|,  &&  & & \gamma = 2, 
    \end{aligned}
    \end{cases} \\ 
    \left|\frac{1}{T}\int_0^T \frac{dx}{\mathbf{U}(\xi x)^{1/2}} - \int_{\T^2} \frac{d\x}{\mathbf{U}(\x)^{1/2}}\right| &\leq 
        \begin{cases}
        \begin{aligned}
        & \textstyle CT^{-\tau}, \qquad\qquad\qquad\,  && \textstyle \tau = \frac{1}{2}\frac{2-\gamma}{2+\gamma}, & & \gamma \in [1,2),  \\      
        & \textstyle CT^{-\tau}, && \textstyle \tau = \frac{\gamma}{1+\gamma}\frac{2-\gamma}{2+\gamma}, & & \gamma \in \textstyle (\frac{2}{3},1), \\      
        & \textstyle CT^{-1/5}|\log(T)|, &&  & & \gamma = \textstyle \frac{2}{3}, \\      
        & \textstyle CT^{-\tau}, && \textstyle \tau =  \frac{1}{2}\frac{\gamma}{1+\gamma}, & & \gamma \in \textstyle (0, \frac{2}{3}) .
        \end{aligned}
        \end{cases}
\end{aligned}
\end{equation*}
Furthermore, the result for $\gamma>2$ holds for all non-resonant $\xi\in \R^2$ with $\sigma_\xi < 2$.
\end{proposition}

The result can also be drawn to \ref{itm:A2}, provided we replace the integral on $[a,a+T]$ such that it is well-defined (avoiding the singularity). Besides, Proposition \ref{prop:ASingleCaseOfApplicationErgodic} can be  be generalized to arbitrary $\mathbf{U}$ satisfying \ref{itm:A0} without additional difficulties (see Remark \ref{rmk:GeneralErgodic} for a general procedure).

\subsection{Organization of the paper}

We begin with background on almost periodic functions and the key ergodic estimate in Section \ref{Sect:preliminaries-almost-periodic}. In Section \ref{Sec:effectiveHamiltonian}, we discuss the representation and regularity of the effective Hamiltonian. Section \ref{Sect:rate2D} combines the findings to prove the Theorem \ref{thm:rate-C1-H-bar-prototype} for potential of prototype \ref{itm:A1} or \ref{itm:A2}, and provides a sketch of the proof for Theorem \ref{thm:n-frequency-lower-bound}. In Section \ref{sec:rateND}, we improve the ergodic estimates for critical cases of characteristics, and apply these improvements to quantify the convergence of Birkhoff average  as seen in Proposition \ref{prop:ASingleCaseOfApplicationErgodic}, along with discussions and open questions.

\section{Almost periodic functions and quantitative ergodic estimates}\label{Sect:preliminaries-almost-periodic} 

\subsection{Notations and definitions} 
\label{Section:Appendix}
We list some function spaces that we use throughout the paper for reference. We denote by $\mathrm{C}^0(\mathbb{R}) = \mathrm{C}(\mathbb{R})$ the set of continuous function on $\mathbb{R}$, and $\mathrm{C}^k(\mathbb{R})$ ($k\in \mathbb{N}$) the set of function such that its $k$-th derivative $f^{(k)} \in \mathrm{C}(\mathbb{R})$. For $\alpha\in (0,1)$, the H\"older space $\C^{0,\alpha}(\R)$ is the set of continuous functions $f$ such that 
    \begin{equation*}
        [f]_{\C^{0,\alpha}(\R)}:= 
        \sup_{x\in \R} \sup_{h\neq 0} \frac{|f(x+h)-f(x)|}{|h|^\alpha} < \infty,
    \end{equation*}
    with $\Vert f\Vert_{\C^{0,\alpha}(\R)}:= \Vert f\Vert_{\C^{1}(\R)} + [f]_{\C^{0,\alpha}(\R)}$. For $\kappa = (\kappa_1,\ldots,\kappa_n)\in \mathbb{Z}^n$ and $\alpha = (\alpha_1,\ldots, \alpha_n) \in (\mathbb{Z}_{\geq 0})^n$ is a multi-index, we define $|\kappa| = \big(|\kappa_1|^2+\ldots+|\kappa_n|^2\big)^{1/2}$ and $\kappa^\alpha = \kappa_1^{\alpha_1} \ldots \kappa_n^{\alpha_n}$. 

\begin{definition} For $f\in L^1(\mathbb{T}^n)$, we define the Fourier transform $\mathcal{F}(f) = \widehat{f}: \mathbb{Z}^n \to \mathbb{C}$ by \begin{equation*}
    \widehat{f}(\kappa) = \int_{\mathbb{T}^n} f(\x)e^{-2\pi i \kappa \cdot \x}\;d\x, \qquad \kappa \in \mathbb{Z}^n.
\end{equation*}
The Fourier series of $f$ is defined by $\sum_{\kappa\in \mathbb{Z}^n} \widehat{f}(\kappa) e^{2\pi i \kappa\cdot \x}$ for $\x\in \mathbb{T}^n$
as a well-defined function in $L^2(\mathbb{T}^n)$. The Sobolev space $H^{s}(\mathbb{T}^n)$ is defined by 
\begin{equation*}
    H^{s}(\mathbb{T}^n) = \left\lbrace f\in \mathcal{D}'(\mathbb{T}^n): (1+|\kappa|^2)^{s/2}|\widehat{f}(\kappa)| \in \ell^2(\mathbb{Z}^n) \right\rbrace,      
\end{equation*}
with $\Vert f\Vert_{H^s} = \left(\sum_{\kappa\in \mathbb{Z}^n} (1+|\kappa|^2)^{s/2} |\widehat{f}(\kappa)|^2\right)^{1/2}$, where $\mathcal{D}'(\mathbb{T}^n)$\label{def:SobolevHs} is the set of distributions on $\mathbb{T}^n$. 
\end{definition}    

For $k\in \mathbb{N}$, let $W^{k,p}(\mathbb{T}^n) = \left\lbrace f\in L^p(\mathbb{T}^n): \partial^\alpha f\in \ell^p(\mathbb{Z}^n)\;\text{for}\;|\alpha|\leq k \right\rbrace$, where $\partial^{\alpha}f$ is the distributional derivative of $f$, and $ \Vert f\Vert_{W^{k,p}(\mathbb{T}^n)} = \sum_{|\alpha|\leq k} \Vert D^\alpha f\Vert_{L^p(\mathbb{T}^n)}$. It is known that if $k\in \mathbb{N}$ then $H^k(\T^n)$ is equivalent to $W^{k,2}(\T^n)$. 
In what follows, we list some useful properties of $H^s(\T^n)$ that we will be using.

\begin{lemma}[Interpolation]\label{prop:Holder-Sobolev} Let $0<s_0<s_1$ and $f\in H^{s_1}(\mathbb{T}^n)$. If $ s = \theta s_1 + (1-\theta)s_0 \in (s_0,s_1)$ for some $\theta \in (0,1)$, then $\Vert f \Vert_{H^{s}(\mathbb{T}^n)}  
    \leq 
    \Vert f \Vert_{H^{s_1}(\mathbb{T}^n)}^\theta 
    \Vert f \Vert_{H^{s_0}(\mathbb{T}^n)}^{1-\theta}$. 
\end{lemma}

\begin{proposition}[Uniform convergence]\label{prop:uniform-convergence-Fourier-series} Let $f\in H^{s}(\mathbb{T}^n)$ for some $s>\frac{n}{2}$. Then the Fourier series of $f$ converges to $f$ with respect to the norm of $\mathrm{C}^0(\mathbb{T}^n)$.
\end{proposition}

\begin{proof} For $N \in \mathbb{N}$, we define $f_N(\x) = \sum_{|\kappa| \leq N} a_\kappa e^{2\pi i \kappa \cdot \x}$ for $\x\in \mathbb{T}^n$.  By the Sobolev Embedding Theorem \color{black} {\cite[Theorem 9.17]{folland_book_real_analysis}} we have $\sum_{\kappa \in \mathbb{Z}^n} |\widehat{f}(\kappa)| < \infty$. Therefore $\{f_N\}_{N\in \mathbb{N}}$ is a Cauchy sequence with respect to $\mathrm{C}^0(\mathbb{T}^n;\mathbb{C})$ norm. By the completeness of $\mathrm{C}^0(\mathbb{T}^n;\mathbb{C})$ there exists $g\in \mathrm{C}^0(\mathbb{T}^n;\mathbb{C})$ such that $\lim_{N\to \infty} \Vert f_N - g\Vert_{\C^{0}(\mathbb{T}^n)} = 0$. It is clear that $\widehat{g}(\kappa) = a_\kappa$ for $\kappa \in \mathbb{Z}^n$, hence $f=g$ and thus $f_N(\x)$ converges to $f(\x)$ uniformly in $\mathrm{C}^0(\mathbb{T}^n)$. 
\end{proof}

\begin{remark}\label{remark:embedding}
The Sobolev Embedding Theorem relates $H^s(\T^n)$ with $\mathrm{C}^{0,\alpha}(\T^n)$ by the following
\begin{equation*}
    H^s(\T^n)\subset \mathrm{C}^{0,\alpha}(\T^n),\quad \text{where }\;\;0<\alpha < \min \lbrace s-{n}/{2}, 1\rbrace.
\end{equation*}
In subsection \ref{ss:ergo-est}, we will see how these different spaces lead to different convergence rates of the Birkhoff average, influencing the outcomes of \ref{itm:P1}---\ref{itm:P3} in Theorem \ref{thm:n-frequency-lower-bound}. 
\end{remark}

\subsection{Almost periodic functions} We follows \cite{bohr_almost_periodic_1947} with the following Definition of almost periodic functions. For $\varepsilon>0$ and $f:\R\to \R$, we say that $\tau$ is an {\it $\varepsilon$-period} of $f$, if $|f(x+\tau) - f(x)| < \varepsilon$ for all $x\in \R$.

\begin{definition} We say that $f:\R\to\R$ is \emph{uniformly almost periodic} if for any $\varepsilon>0$, the set of $\varepsilon$-periods
\begin{equation*}
	E(\varepsilon, f) = \{\tau\in \mathbb{R}:  |f(x+\tau) - f(x)|<\varepsilon\;\text{for all}\;x\in \R\}\;\text{is \emph{relatively dense} in}\;\R.
\end{equation*}
In other words, there exists $l_\varepsilon(f) > 0$ such that every interval of the form $[x,x+l_\varepsilon(f)]$ contains a $\varepsilon$-periods. We say that $l_\varepsilon(f)$ is an \emph{inclusion interval length} of $E(\varepsilon,f)$. The set of such functions is denoted by $\mathrm{AP}(\mathbb{R})$.
\end{definition}
Notice that $\mathrm{AP}(\R)\subset \mathrm{BUC}(\R)$ (see \cite{Katznelson2004}). Given $f\in \mathrm{BUC}(\R)$, an equivalent definition (\cite[Theorem 5.5]{Katznelson2004}) of almost periodicity is the \emph{uniformly normal} property given by Bochner, i.e., the family of functions $\{f(\cdot + h): h\in \mathbb{R}\}$ is relatively compact in $\mathrm{BUC}(\mathbb{R})$ (see \cite{armstrong_error_2014, Ishii_almost_periodic_1999, tran_hamilton-jacobi_2021}). Some important properties of almost periodic functions include the following.

\begin{lemma}[{\cite[Section 5]{Katznelson2004}}] If $f\in \mathrm{AP}(\mathbb{R})$ then the mean value $\mathcal{M}(f)$ of $f$ exists, and
\begin{align*}
    \mathcal{M}(f) 
    &= \lim_{T\rightarrow \infty} \frac{1}{2T}\int_{-T}^{T} \;f(x)\;dx 
     = \lim_{T\rightarrow -\infty} \frac{1}{T}\int_{T}^0 f(x)\;dx = \lim_{T\rightarrow +\infty} \frac{1}{T}\int_{0}^T f(x)\;dx.
\end{align*}
Furthermore, if $f\geq 0$ then $\mathcal{M}(f) > 0$ except $f\equiv 0$.
\end{lemma}

\begin{lemma}[{\cite[Theorem 1.4 and p. 14]{besicovitch1954almost}}]\label{prop:ap:uap-ergodic-estimate-eps-l-eps} Let $f\in \mathrm{AP}(\mathbb{R})$, $\varepsilon>0$ and $l_\varepsilon(f)$ be the length of the inclusion interval. Then $\sup_{\mathbb{R}} |f| \leq \sup_{[0,l_\varepsilon(f)]}  |f| + \varepsilon$ and furthermore, for $T>0$ we have
\begin{equation*}
    \left|\frac{1}{T}\int_0^T f(x)\;dx - \mathcal{M}(f)\right|   \leq  \varepsilon + 2\Vert f\Vert_{L^\infty(\mathbb{R})} \frac{ l_\varepsilon(f)}{T}.
\end{equation*}
\end{lemma}

\subsection{Convergence rate of Birkhoff average for bounded quasi-periodic functions}\label{ss:ergo-est}

Recall that for periodic function $f\in \mathrm{BUC}(\mathbb{R})$ with a period $L$, $\mathcal{M}(f) = \frac{1}{L}\int_0^L f(x)dx$. 
Due to \cite[Lemma 4.2]{mitake_tran_yu_2019_rate_of_convergence_ARMA} or \cite{Neukamm2017, tu_2018_rate_asymptotic}, we have
\begin{equation*}
    \left|\frac{1}{T}\int_0^{T} f(x)\;dx - \mathcal{M}(f)\right| \leq \frac{\mathcal{M}(f)}{T},
\end{equation*}
where the right-hand side is optimal. In the quasi-periodic setting, such a  rate of convergence is not always available, unless certain conditions on the non-resonance of $\xi$ and the regularity of the potentials are supplied. For this purpose, the following Lemma is needed:

\begin{lemma}\label{prop:ap:QP-implies-AP-and-sup-inf} Let $\F\in \mathrm{C}(\mathbb{T}^n;\mathbb{C})$ and $\xi = (\xi_1,\ldots,\xi_n)\in \mathbb{R}^n$ be nonresonant. Then $f(x) = \F(\xi x)\in \mathrm{AP}(\mathbb{R})$ with $\sup_{\mathbb{R}} |f| = \max_{\mathbb{T}^n} |\F|$ and $\inf_{\mathbb{R}} |f| = \min_{\mathbb{T}^n} |\F|$. Furthermore $\mathcal{M}(f) = \int_{\mathbb{T}^n} \F(\x)\;d\x$. 
\end{lemma}

\begin{proof} The facts that $\sup_{\mathbb{R}} |f| = \max_{\mathbb{T}^n} |\F|$ and $\inf_{\mathbb{R}} |f| = \min_{\mathbb{T}^n} |\F|$ are consequence of Kronecker's Theorem (see \cite{Katznelson2004}). 
If $\F\in \C(\mathbb{T}^n)$, for each $\varepsilon > 0$,  there exists a finite subset $S_\varepsilon\subset \mathbb{Z}^n$ and a trigonometric polynomial $ P_{\varepsilon}(\textbf{x}) = \sum_{\kappa\in S_\varepsilon} a_\kappa \, e^{2\pi i \kappa\cdot \textbf{x}}$ such that $\Vert \F-P_\varepsilon\Vert_{\mathrm{C}(\mathbb{T}^n)} < \varepsilon$ and $a_\textbf{0} = \int_{\mathbb{T}^n} \F(\x)\;d\x$. Indeed, $P_\varepsilon$ can be chosen to be the truncated of the Fourier series $\F*\eta_\varepsilon$ where $\eta_\varepsilon$ be the standard mollifiers. We have
\begin{align*}
    &\lim_{T\to\infty} \frac{1}{T}\int_0^T P(\xi x)\;dx = \lim_{T\to\infty} \frac{1}{T}\int_0^T \sum_{\kappa\in S_\varepsilon} a_{\kappa}\, e^{2\pi i \kappa\cdot \xi x}\;dx = a_0 + \lim_{T\to\infty} \frac{1}{T} \sum_{\kappa\in S_\varepsilon\backslash\{0\}} a_\kappa\frac{e^{2\pi i \kappa\cdot \xi T}-1}{2\pi i \kappa\cdot\xi} = a_0.
\end{align*}
Thus
\begin{equation*}
    \frac{1}{T}\int_0^T f(x)\;dx - a_\textbf{0} = \frac{1}{T}\int_0^T \Big(\F(\xi x)-P_\varepsilon(\xi x)\Big)\;dx + \left(\frac{1}{T}\int_0^T P_\varepsilon(\xi x)\;d x - a_\textbf{0}\right).
\end{equation*}
Hence
\begin{equation*}
    \limsup_{T\to\infty}\left|\frac{1}{T}\int_0^T f(x)\;dx - \int_{\mathbb{T}^n} \F(\x)\;d\x\right| 
    \leq 
    \Vert \F-P_\varepsilon\Vert_{\mathrm{C}^0(\mathbb{T}^n)} + \lim_{T\to\infty}\left|\frac{1}{T}\int_0^T P_\varepsilon(\xi x)\;dx - a_\textbf{0}\right| <\varepsilon.
\end{equation*}
Since $\varepsilon$ is chosen arbitrary we obtain the result.
\end{proof}

\begin{proposition}\label{prop:ap:rate-Mean-QR} Let $\xi = (\xi_1,\ldots,\xi_n) \in \mathbb{R}^n$ be a Diophantine frequency with the index $\sigma_\xi$ and $\F\in H^{s}(\mathbb{T}^n)$ for some $s > \frac{n}{2} + \sigma_\xi$. There exists $C = C(n,s)$ such that, for every $T>0$ we have
\begin{equation*}
    \left|\frac{1}{T}\int_{0}^T \F(\xi x)\;dx - \int_{\mathbb{T}^n} \F(\x)\;d\x \right| \leq \frac{C(n, s)\Vert \F\Vert_{H^{s}(\mathbb{T}^n)}}{T}.
\end{equation*} 
\end{proposition}

\begin{proof} For $N\in \mathbb{N}$ we define $\F_N(x) = \sum_{|\kappa| \leq N} \widehat{\F}(\kappa) e^{2\pi i \kappa \cdot x}$ for $x\in \mathbb{T}^n$. Since $\F_N\to \F$ uniformly in $\mathrm{C}^0(\mathbb{T}^n)$ (see Proposition \ref{prop:uniform-convergence-Fourier-series}), $\sup_{x\in \mathbb{R}} |\F_N(\xi x) - \F(\xi x)| \leq \sup_{\x\in \mathbb{T}^n} |\F_N(\x) - \F(\x)| \to 0$ as $N\to \infty$. We have 
\begin{align*}
    \frac{1}{T}\int_0^T \F_N(\xi x)\;dx 
    &= 
    \frac{1}{T} \int_0^T \sum_{|\kappa|\leq N} \widehat{\F}(\kappa) e^{2\pi i \kappa \cdot \xi x}\;dx
    = \int_{\mathbb{T}^n} \F(\textbf{x})\;d\textbf{x} + \frac{1}{T}\sum_{0<|\kappa|\leq N} \widehat{\F}(\kappa) 
    \left(
        \frac{e^{2\pi i \kappa \cdot \xi T} - 1}{2\pi i \kappa \cdot \xi}
    \right).
\end{align*}
From the Diophantine condition of $\xi$, for any $s>0$ that
\begin{align*}
    &\sum_{0<|\kappa|\leq N} \big|\widehat{\F}(\kappa)\big| 
    \left|
        \frac{e^{2\pi i \kappa \cdot \xi T} - 1}{2\pi i \kappa \cdot \xi}
    \right| 
    \leq 
    \frac{1}{\pi C(\xi)}
    \sum_{0<|\kappa|\leq N} |\kappa|^\sigma \big|\widehat{\F}(\kappa)\big|   
    \leq \frac{1}{\pi C(\xi)}\sum_{0<|\kappa|\leq N}  
    \left(|\kappa|^{-\frac{n+s}{2}}\right)\left( |\kappa|^{\frac{n+s}{2} + \sigma} |\widehat{\F}(\kappa)| \right)\\
    &\qquad\qquad\qquad 
        \leq \frac{1}{\pi C(\xi)} 
    \left(\sum_{0<|\kappa|\leq N} \frac{1}{|\kappa|^{n+s}} \right)^{1/2}
    \left(\sum_{0<|\kappa|\leq N} |\kappa|^{n+s+2\sigma} |\widehat{\F}(\kappa)|^2\right)^{1/2} 
    \leq \frac{C(n,s)}{\pi C(\xi)} \Vert \F\Vert_{H^{\frac{n+s}{2}+\sigma}(\mathbb{T}^n)}.
\end{align*}
Therefore, we have 
\begin{align*}
    \left|\frac{1}{T}\int_0^T \F_N(\xi x) \;dx - \int_{\mathbb{T}^n} \F(\textbf{x})\;d\textbf{x}\right| \leq \frac{1}{T}\left( \frac{C(n,s)}{\pi C(\xi)} \right)\Vert \F\Vert_{H^{\frac{n+s}{2}+\sigma}(\mathbb{T}^n)}. 
\end{align*}
As $N\to \infty$ we obtain the conclusion.
\end{proof}

In view of  Remark \ref{remark:embedding}, we now present results concerning the case $\F\in \C^{0,\alpha}(\T^n)$.

\begin{lemma}\label{lem:P4RateSec2} Let $n\geq 2$ and $\omega = (\omega_1,\ldots, \omega_{n-1}, 1) \in \R^{n}$ be a non-resonant vector with $\omega_i$ is algebraic irrational for $i=1,\ldots, n-1$. Let $\F\in \C^{0,\alpha}(\T^{n})$ and $f(x) = \F(\omega x)$ for $x\in \R$. 
For any $0 < \nu \ll 1$, there exists $C_\nu = C(n, \omega, \alpha, \nu, \F)$ such that $l_\varepsilon(f) \leq C_\nu \varepsilon^{-\frac{n-1}{\alpha} - \nu}$. 
\end{lemma}

\begin{proof} Let $C_\F$ be the H\"older constant of $\F$, in the sense that, if $\mathbf{x} = (x_1,\ldots, x_{n}), \mathbf{y} = (y_1,\ldots, y_{n}) \in \T^{n}$ then $|\F(\x) - \F(\mathbf{y})| \leq C_\F\sum_{i=1}^{n}|x_i-y_i|^\alpha$. For $\omega \in \mathbb{R}\backslash \mathbb{Q}$ and $q\in \mathbb{Z}\backslash \{0\}$, we define $\Vert \omega \cdot q\Vert = \inf_{\kappa\in \mathbb{Z}} |\omega\cdot q - \kappa|$. To simplify the notation further, we denote $m = n - 1$ so that $\omega = (\omega_1,\ldots, \omega_{n-1},1) = (\omega_1,\ldots, \omega_{m},1) \in \R^{m+1}$. 
We claim there are only finitely many solutions $(q, R) \in \mathbb{Z} \times \mathbb{R}^+$ to
\begin{equation*}
    \sup_{i=1,\ldots, n}\Vert \omega_i \cdot q\Vert < R^{-\frac{1}{m} - \delta} \qquad\text{and}\qquad 
    |q| < R^{\frac{1}{\left(1+m\delta\right)}}.
\end{equation*}
If the reverse is true, there exist infinitely many solutions $q \in \mathbb{Z}$ to the equation:
\begin{equation*}
     \sup_{i=1,\ldots, m}\Vert \omega_i \cdot q\Vert  < R^{-\left(\frac{1}{m}+\delta\right)} < |q|^{-\left(\frac{1}{m}+\delta\right)\left(1+m\delta\right)} = |q|^{-\left(\frac{1}{m}+\theta\right)}
\end{equation*}
for some $\theta > 0$. This contradicts Schmidt's Theorem \cite{schmidt_simultaneous_1970}, confirming our claim. By \cite[Theorem VI p. 82]{cassels_1957}, we can find $\gamma = \gamma(\omega)>0$, such that, to $(\beta_1,\ldots, \beta_m)\in\R^n$ and $R>0$, there exists $q \in \mathbb{N}$ that solves
\begin{equation}\label{eq:SolutionOfQ}
    \sup_{i=1,\ldots, m}\Vert \omega_i \cdot q + \beta_i \Vert 
    < \gamma^{-1} R^{\frac{(m-1)(1+m\delta)^2-m}{m(1+m\delta)}}, \qquad \text{and}\qquad |q| < |R|^{1+m\delta}.
\end{equation}
Choose $\delta > 0$ small such that $(m-1)(1+m\delta)^2 < m$. Then for $\varepsilon > 0$ there exists $R = R(\delta, \varepsilon)$ such that
\begin{equation}\label{eq:DefinitionOfRWithEpsilon}
    \Big(\gamma^{-1} R^{\frac{(m-1)(1+m\delta)^2-m}{m(1+m\delta)}} \Big)^\alpha = \frac{\varepsilon}{m C_\F}. 
\end{equation}
Let $T_\varepsilon = R + R^{1+m\delta} + 1$. For $a\in \R$, we show that $[a,a+T_\varepsilon]$ contains an $\varepsilon$-period of $f$. Let $k = [a+R+1] \in \mathbb{Z}$ be the integer part of $a+R+1$, and $\beta_i = \omega_i \cdot k$ for $i=1,\ldots, m$, then, we can obtain $q\in \mathbb{N}$ be a solution to \eqref{eq:SolutionOfQ} for the given $(\beta_1,\ldots, \beta_m)$. Finally, we define $\ell = k + q \in \mathbb{Z}$, then $\ell \in [a, a+T_\varepsilon]$. We have
\begin{align*}
    |f(x+\ell) - f(x)| 
    &= C_\F \sum_{i=1}^m \Vert \omega_i \cdot q + \beta_i \Vert ^\alpha
    \leq C_\F \sum_{i=1}^m  \Big(\gamma^{-1} R^{\frac{(m-1)(1+m\delta)^2-m}{m(1+m\delta)}} \Big)^\alpha  = \varepsilon
 \end{align*}
thanks to \eqref{eq:DefinitionOfRWithEpsilon}. Hence, $\ell\in [a,a+T_\varepsilon]$ is an $\varepsilon$-period for $f$. We conclude that
\begin{equation*}
    l_\varepsilon(f) \leq T_\varepsilon \leq C \varepsilon
    ^{-\frac{m}{\alpha} \cdot \frac{1+m\delta}{m - (m-1)(1+m\delta)^2}} = C\varepsilon^{-\frac{m}{\alpha}(1+\mathcal{O}(\delta))}
\end{equation*}
where $C = C(m, \omega, \delta, C_\F, \alpha)$. For $0<\nu\ll 1$, we can find $\delta = \delta(m, \alpha, \nu)$ and a constant $C = C(m, \omega, C_\F, \alpha, \nu)$ such that $l_\varepsilon(f) \leq C \varepsilon^{-\frac{m}{\alpha} - \nu}$, thus the proof is complete.
\end{proof}

\begin{lemma}[$n$-frequency {\cite[Theorem 2.1]{Bryan1998}}]\label{coro:l-eps-nD} Let $n\geq 2$. We define
\begin{equation}\label{eq:defEnSetFullMeasure}
    \mathcal{E}_{n-1}:= 
    \left\lbrace
    \begin{aligned}
    & \omega \in \R^{n-1}: 
    \;\text{there is only finitely many}\; (q,R) \in \mathbb{Z} \times \R^+\;\text{such that} \\
    & \qquad \Vert \omega_i q\Vert:= \inf_{\kappa \in \mathbb{Z}} |\omega_i q - \kappa| < R^{-\frac{1}{n-1}} \;\text{for}\;i=1,\ldots, n-1 \;\text{and}\;
        |q| <  R|\log(R)|^{-2}.
    \end{aligned}
    \right\rbrace,
\end{equation}
and 
\begin{equation}\label{eq:defOmeganFullMeasure}
    \Omega_n = \Big\{t(\omega_1,\ldots, \omega_{n-1}, 1): t\in \mathbb{R}\backslash\{0\}\; \text{and}\;(\omega_1,\ldots, \omega_{n-1})\in \mathcal{E}_{n-1} \Big\}.
\end{equation}
Then $\Omega_n\subset\R^n$ is of full measure. Let $f(x) = \F(\xi x)$ for $x\in\R$, where $\xi\in \Omega_{n}$ and $\F\in \C^{0,\alpha}(\T^n)$. There exists $C=C\big(n, \xi, \alpha, \Vert \F\Vert _{\mathrm{C}^{0,\alpha}(\T^n)}\big)$ such that $l_\varepsilon(f) \leq C|\log(\varepsilon)|^{3(n-1)} \varepsilon^{-\frac{n-1}{\alpha}}$. 
\end{lemma}


\begin{lemma}[2-frequency {\cite[Theorem 3]{naito_1996}}]\label{prop:l-eps-2D} Let $\F\in \mathrm{C}^{0,\alpha}(\T^2)$, $\xi \in \R^2$ be non-resonant, and $f(x) = \F(\xi x)$ for $x\in \R$. If $\sigma_{\xi} = 1$, then $l_\varepsilon(f) \leq C\varepsilon^{-1/\alpha}$ where $C$ depends on $\xi, \alpha$ and $\Vert \F\Vert _{\mathrm{C}^{0,\alpha}}$. 
\end{lemma}

\begin{proposition}[General $n$-frequency quasi-periodic potential]\label{prop:n-frequency-rate} Let $\F\in \C^{0,\alpha}(\T^n)$ with $n\geq 2$. 
\begin{itemize}
    \item[$\mathrm{(i)}$] For almost every $\xi \in \R^n$, i.e., $\xi \in \Omega_n$ defined in \eqref{eq:defOmeganFullMeasure}, there exists $C=C(n, \xi, \Vert \F\Vert_{\mathrm{C}^{0,\alpha}})$ such that 
    \begin{equation*}
        \left|\frac{1}{T}\int_0^T \F(\xi x)\;dx - \int_{\T^n} \F(\x)\;d\x\right| \leq C |\log(T)|^{3(n-1)} \left(\frac{1}{T}\right)^{\frac{\alpha}{\alpha+n-1}}. 
    \end{equation*}
    
    \item[$\mathrm{(ii)}$] Let $\xi=(\xi_1,\ldots, \xi_n)\in \R^n$ be non-resonant with $\xi_i$ is algebraic for $i=1,\ldots,n$. For every $0<\nu \ll 1$ there exists
    $C=C(n, \xi, \Vert \F\Vert_{\C^{0,\alpha}}, \nu)$ such that 
    \begin{equation*}
        \left|\frac{1}{T}\int_0^T \F(\xi x)\;dx - \int_{\T^n} \F(\x)\;d\x\right| \leq C_\nu \left(\frac{1}{T}\right)^{\frac{\alpha}{\alpha+n-1 + \nu}}.
    \end{equation*}
    
    \item[$\mathrm{(iii)}$] If $n=2$ and $\xi = (\xi_1,\xi_2)$ and Diophantine index $\sigma_\xi = 1$, then there exists $C = C(\xi, \Vert \F\Vert_{\C^{0,\alpha}})$ such that
    \begin{equation*}
        \left|\frac{1}{T}\int_0^T \F(\xi x)\;dx - \int_{\T^n} \F(\x)\;d\x\right| \leq 
        C\left(\frac{1}{T}\right)^{\frac{\alpha}{\alpha+1}}.
    \end{equation*}
\end{itemize}

\end{proposition}
\begin{proof} For (i), by Lemma \ref{coro:l-eps-nD}, we have $l_\varepsilon(f) \leq C |\log(\varepsilon)|^{3(n-1)} \varepsilon^{-\frac{n-1}{\alpha}}$. By Lemma \ref{prop:ap:uap-ergodic-estimate-eps-l-eps} we have
\begin{equation*}
    \left|\frac{1}{T}\int_0^T \F(\xi x)\;dx - \int_{\mathbb{T}^n} \F(\x)\;d\x  \right| \leq \varepsilon + \Vert \F\Vert_{L^\infty(\mathbb{T}^2)}\frac{C(\xi)|\log(\varepsilon)|^{3(n-1)}}{\varepsilon^{(n-1)/\alpha} T}.
\end{equation*}
Choose $\varepsilon = T^\theta$ where $\theta = \frac{\alpha}{\alpha+n-1}$ we obtain the conclusion. The proofs for (ii) and (iii) follow similarly, by adapting Lemmas \ref{lem:P4RateSec2} and \ref{prop:l-eps-2D}, respectively. 
\end{proof}

\section{The effective Hamiltonian}\label{Sec:effectiveHamiltonian}

\subsection{A representation for the effective Hamiltonian}

\begin{lemma}\label{lem:basic-properties-H-bar} Assume \ref{itm:A0}. 
\begin{itemize}
\item[$\mathrm{(i)}$] $\overline{H}$ is convex, Lipschitz, coercive and even. Moreover, $\overline{H}\geq 0$ in $\R$ and $\overline{H}(0) = 0$. Let us define
\begin{equation}\label{eq:defn-p0}
    p_0 = \int_{\mathbb{T}^n} \left(2\mathbf{U}(\x)\right)^{1/2}\;d\x = \mathcal{M}\left(\sqrt{2\mathbf{U}(\xi x)}\right).
\end{equation}   
Then $\overline{H}(p) = 0$ for $|p|\leq |p_0|$, and for $p \geq |p_0|$, the value $\overline{H}(p)$ can be established by
    \begin{equation}\label{eq:effectiveH:representation}
        p = \int_{\mathbb{T}^n} \left( 2(\overline{H}(p) + \mathbf{U}(\x)) \right)^{1/2}\;d\x.
    \end{equation}
\item[$\mathrm{(ii)}$] If $|p|\geq |p_0|$ then there exists a unique (up to adding a constant) exact sublinear corrector, i.e., a classical solution $v_p\in \mathrm{C}^1(\mathbb{R})$ to the cell problem \eqref{eq:CPdelta} with $\delta = 0$: 
    \begin{equation}\label{eq:corrector}
        \frac{1}{2}|p+v'_p(x)|^2 + V(x) = \overline{H}(p) \qquad\text{in}\;\mathbb{R}. 
        \tag{CP$_0$}
    \end{equation}
\item[$\mathrm{(iii)}$]  $\overline{H}$ is strictly increasing (resp. decreasing) on $[p_0,\infty)$ (resp. $(-\infty, -p_0] $).
\end{itemize}
\end{lemma}

\begin{proof} The proof of (i) is standard and resembles the periodic case, thus we omit it and refer to \cite[Section 5.5.2]{tran_hamilton-jacobi_2021}. For (ii), we only need to consider $p\geq p_0$ since $\overline{H}$ is even. We construct an exact corrector $v_p\in \mathrm{C}^1(\mathbb{R})$ to \eqref{eq:corrector}.  Let $\mu = \overline{H}(p) \geq 0$ for $p\geq p_0$, let 
    \begin{equation*}
        w_p(t) = \int_0^t \sqrt{2(\mu + \mathbf{U}(\xi \x))}\;dx, \qquad t\in \mathbb{R}. 
    \end{equation*}
    Since $x\mapsto (\mu+ \mathbf{U}(\xi \x))^{1/2}$ is quasi-periodic, by Lemma \ref{prop:ap:QP-implies-AP-and-sup-inf} we have
    \begin{equation}\label{eq:effectiveH:representation-mean}
        p = \lim_{t\to \infty} \frac{w_p(t)}{t} = \lim_{t\to \infty} \frac{1}{t}\int_0^t  \sqrt{2(\mu + \mathbf{U}(\xi \x))}\;dx=\int_{\mathbb{T}^n} \sqrt{2(\mu + \mathbf{U}(x))}\;dx. 
    \end{equation}
   Since $w_p\in \mathrm{C}^1(\mathbb{R})$, the function 
\begin{equation}\label{eq:constructionvp}
	v_{p} (x) =   w_p(x) - p x
\end{equation}   
    belongs $\mathrm{C}^1(\mathbb{R})$ and is sublinear due to \eqref{eq:effectiveH:representation-mean}, and the \eqref{eq:corrector} follows since 
    \begin{equation*}
        \frac{1}{2}|p + Dv_{p}(x)|^2 + V(x) = \frac{1}{2}|Dw_\mu(x)|^2 + V(x) = \overline{H}(p)
    \end{equation*}
    in $\mathbb{R}$. Furthermore, $\mu\mapsto p$ as in \eqref{eq:effectiveH:representation-mean} is continuous and increasing from $[0,\infty) \to [p_0, \infty)$, thus it admits a unique continuous inverse map $p\mapsto \overline{H}(p):[p_0,\infty)\to [0,\infty)$ satisfying \eqref{eq:effectiveH:representation}. If $|p|\leq |p_0|$ then by convexity we have $\overline{H}(p) = 0$ since $\overline{H}(-p_0) = \overline{H}(-p_0) = 0$. 
\end{proof}

\begin{remark} Analogues of  \eqref{eq:defn-p0} and \eqref{eq:effectiveH:representation} are available for the almost periodic setting in the same manner.
\end{remark}

\begin{lemma}[Regularity of $\overline{H}$]\label{prop:D1H-bar-derivative} Assume \ref{itm:A0}. Then $\overline{H}\in \mathrm{C}^1(\mathbb{R}\backslash [-p_0, p_0])$, where $p_0$ is defined in \eqref{eq:defn-p0} with 
\begin{equation}\label{eq:representation-derivative-formula}
    \frac{1}{\overline{H}'(p)} = \int_{\mathbb{T}^2}  
    \frac{d\x}{\left(2(\overline{H}(p) + \mathbf{U}(\x))\right)^{1/2}}
    \;d\x, \qquad |p| > |p_0|.
\end{equation}
In particular, we have $\overline{H}\in \mathrm{C}^1(\mathbb{R})$ if and only if $\int_{\T^n} \mathbf{U}(x)^{-1/2}\;d\x = \infty$. 
\end{lemma}

\begin{proof} We only need to consider $p>p_0$. We denote $\phi: [0,\infty)\rightarrow[p_0,\infty)$ by
\begin{equation}\label{eq:defn-F-mu}
    \phi(\mu) := p = \int_{\mathbb{T}^n}(2(\mu+\mathbf{U}(\x)))^{1/2}\;d\x = \mathcal{M}(\F_\mu), \qquad\text{where}\; \F_\mu = \sqrt{2}(\mu + \mathbf{U})^{1/2},
\end{equation}
which is actually the inverse of $\overline{H}:[p_0,\infty)\rightarrow[0,\infty)$. Consequently, $\overline{H}'(p)\phi'(\mu) = 1$ when $\phi$ is differentiable. As we can see,  
\begin{equation}\label{eq:def-derivtive-Df-inverse-H-bar}
    \phi'(\mu) = \int_{\mathbb{T}^2} \left(2(\mu+\mathbf{U}(\x))\right)^{-1/2}\;d\x,
\end{equation}
so \eqref{eq:representation-derivative-formula} can be deduced. In fact, $\overline H$ is infinitely continuously differentiable on $(p_0,\infty)$. For $p>p_0$, we have
\begin{equation*}
    \frac{p - p_0}{\overline{H}(p) - \overline{H}(p_0)}=\frac{\phi(\mu) - \phi(0)}{\mu - 0} = 2\int_{\mathbb{T}^2} \frac{d\x}{\sqrt{2(\mu + \mathbf{U}(\x))} + \sqrt{2\mathbf{U}(\x)}}.
\end{equation*}
If $\mathbf{U}^{-1/2}\notin L^1(\mathbb{T}^2)$ then we can use monotone convergence theorem to obtain
\begin{equation*}
    \phi'(0^+) = \lim_{\mu\to 0^+} \frac{\phi(\mu) - \phi(0)}{\mu - 0} = \int_{\mathbb{T}^2}\frac{d\x}{\sqrt{2\mathbf{U}(\x)}} = +\infty. 
\end{equation*} 
Therefore $\overline{H}'_+(p_0) = 0$ and thus $\overline{H}$ is continuously differentiable at $p_0$.
\end{proof}

Next, we quantify the sublinear property of the corrector $v_p$ constructed in Lemma \ref{lem:basic-properties-H-bar}.

\begin{corollary}[Growth rate of sublinear corrector for $n$-frequency quasi-periodic potential]\label{prop:rate-corrector-sublinear-badly-approximable} Assume \ref{itm:A0} and $n\geq 2$. For $|p_0|\leq |p| \leq M$, let $v_p\in \mathrm{C}^1(\mathbb{R})$ be the exact corrector to \eqref{eq:corrector} with $v_p(0) = 0$. Then there exists $C = C(\mathbf{U}, \xi, M)$ such that
\begin{numcases}
    {\left|\frac{v_p(t)}{t}\right| 
	    \leq }
    C|t|^{-1}
        &\text{under \ref{itm:P1}}, \nonumber \\ 
    C |t|^{-\frac{\alpha}{\alpha+n-1}} |\log(t)|^{3(n-1)}
        &\text{under \ref{itm:P2}}, \nonumber\\ 
    C_\nu |t|^{-\frac{\alpha}{\alpha+n-1 +\nu}} 
        &\text{under \ref{itm:P4New3}}, \label{eq:correctorP4New3}\\ 
    C |t|^{-\frac{\alpha}{\alpha+1}} 
        &\text{under \ref{itm:P3}},\nonumber
\end{numcases}
In \eqref{eq:correctorP4New3}, we mean that for every $0<\nu \ll 1$ there exists a constant $C_\nu$ such that the estimate holds. 
\end{corollary}

\begin{proof} For $|p|\geq |p_0|$, let $\mu = \overline{H}(p) \geq 0$, then $p = p_\mu = \phi(\mu) = \mathcal{M}(\F_\mu)$ as defined in \eqref{eq:effectiveH:representation-mean} and \eqref{eq:defn-F-mu}, where $\F_\mu(\x) = \sqrt{2}(\mu+\mathbf{U}(\x))^{1/2}$. From \eqref{eq:constructionvp} we have
\begin{align}\label{eq:GrowthRateCorrector}
    \left|\frac{v_p(t)}{t}\right| 
    &= \left|\frac{w_p(t)}{t} - p\right| 
    = \left|\frac{1}{t}\int_0^t \F_\mu(\xi x)\;dx - \mathcal{M}(\F_\mu)\right| .
\end{align}
Applying Propositions \ref{prop:n-frequency-rate} and \ref{prop:ap:rate-Mean-QR} we obtain the conclusion.
\end{proof}

\begin{remark} Under \ref{itm:P1}, the exact corrector $v_p$ to \eqref{eq:corrector} is bounded with $v'_p$ being almost periodic, Thus, by applying a theorem of Bohr (see \cite[Theorem 5.20]{Katznelson2004}), we conclude that $v_p$ is almost periodic. 
\end{remark}

\subsection{$\mathrm{C}^{1,\beta}$-regularity of $\overline{H}$ and large time average of characteristics} 

For a Hamiltonian \( H(x,p) : \mathbb{R}^2 \to \mathbb{R} \) that is convex and superlinear in its second argument, its corresponding Lagrangian \( L(x,v) \) is defined via the Legendre transform:  
\begin{equation}\label{eq:LegendreTransform}
    L(x,v) = \sup_{p \in \mathbb{R}} \big(p \cdot v - H(x,p)\big), \qquad (x,v) \in \mathbb{R} \times \mathbb{R}.
\end{equation}  
\color{black}
If $p\in \mathbb{R}$ and $v_p$ is an exact corrector to \eqref{eq:CPdelta}, a path $ \eta_p:[0,\infty)\to \mathbb{R}$ is called a \emph{characteristic}, or a \emph{calibrated curve} associated with $p$, if
\begin{align*}
    \int_{a}^b \Big(L(\eta_p(s), \dot{\eta}_p(s)) + \overline{H}(p)\Big)\;ds = 
    p\cdot \eta_p(b) - p\cdot\eta_p(a) 
        + 
    v_p(\eta_p(b)) - v_p(\eta_p(a))
\end{align*}
for $(a,b)\subset [0,\infty)$. It is know that, upon a subsequence (see \cite{EWeinanAubryMatherCPAM1999, EvansGomesARMAPart12001, fathi_book, gomes_CalVarPDE2002, tran_yu_WeakKAM2022}) we have
\begin{equation}\label{eq:rate:eta-average}
    \lim_{t_i\to \infty }\frac{\eta_p(t_i)}{t_i} = q \in \partial \overline{H}(p),
\end{equation}
where $\partial \overline{H}(p)$ is the subdifferential of $\overline{H}$ at $p$, and $q$ is referred to as a \emph{rotation vector} (in higher dimensions). In particular, if $\overline{H}$ is differentiable at $p$ then the convergence is in the full sequence
\begin{equation}\label{eq:rate:eta-average-full}
    \lim_{t\to \infty }\frac{\eta_p(t)}{t} = \overline{H}'(p).
\end{equation}
As in \cite[eq. (3)]{tran_hamilton-jacobi_2021}, quantifying the rate of convergence of \eqref{eq:rate:eta-average-full} is an open problem. If $\overline{H}$ is $\mathrm{C}^2$ at $p$, a rate $\mathcal{O}(T^{-1/2})$ for \eqref{eq:rate:eta-average-full} is obtained (\cite{tran_hamilton-jacobi_2021, mitake_tran_yu_2019_rate_of_convergence_ARMA, gomes_CalVarPDE2002}). However, under our prototype \ref{itm:A1} $\overline{H}\notin \C^2$ at $\pm p_0$. 
For our purpose of establishing the convergence rate of the homogenization, we need to establish a uniform rate for \eqref{eq:rate:eta-average-full} as $|p|\to |p_0|$, which can be addressed in the following Proposition. We first state a definition first.  For a real-valued function \(f: \R \to \R\) and \(x_0 \in \R\), we denote the one-sided derivatives of \(f\) from the right and left at \(x_0\), respectively, as:  
\begin{equation*}  
    f'_{\pm}(x_0) = \lim_{x \to x_0^{\pm}} \frac{f(x) - f(x_0)}{x - x_0}.  
\end{equation*}  
\color{black}

\begin{proposition}\label{prop:boostrap-rate} Assume \ref{itm:A0} and the following:
\begin{itemize}
    \item[$\mathrm{(i)}$] $\overline{H}$ is $\C^{1,\beta}(\R\backslash [-p_0, p_0])$ for some $\beta \in (0,1]$; 
    \item[$\mathrm{(ii)}$]
    \label{eq:condtion-growth-sublinear} for $|p_0|\leq |p|\leq M$, the exact corrector $v_p$ to \eqref{eq:corrector} satisfies $ \left|\frac{v_p(t)}{t}\right| 
        \leq 
        \frac{C_M}{|t|^\theta}$ for some $\theta \in (0,1]$.
\end{itemize}
Let $|p_0|\leq |p|\leq M$ and $\eta_p:[0,\infty) \to \mathbb{R}$ be a characteristic that corresponds to $p$, then
\begin{equation*}
\begin{aligned}
      &\left|\frac{\eta_p(t)}{t} - \overline{H}'(p)\right| 
      \leq C 
      \left(\frac{1}{|t|}\right)^{\frac{\theta \beta}{1+\beta}} 
      &&\qquad \text{if}\;|p|>|p_0|, \\ 
      &\left|\frac{\eta_{\pm p_0}(t)}{t} - \overline{H}_{\pm}'(\pm p_0)\right| 
      \leq C 
      \left(\frac{1}{|t|}\right)^{\frac{\theta \beta}{1+\beta}} 
      &&\qquad \text{if}\;|p|=|p_0|. 
\end{aligned}
\end{equation*}
\end{proposition}

\begin{proof}[Proof of Proposition \ref{prop:boostrap-rate}] Let $\tilde{p}, p \geq p_0$ and $\tilde{\mu} = \overline{H}(\tilde{p}), \mu = \overline{H}(p) \geq 0 = \overline{H}(p_0)$. 
We observe that, as a characteristic, we have
\begin{equation}\label{eq:bound-eta}
    \left|\frac{\eta(t)}{t}\right| \leq \max_{s\in (0,\infty)} |\dot{\eta}(s)| \leq C_0,
\end{equation}
where $C_0$ depends only on the Hamiltonian. In particular, we can choose $C_0=\sqrt{2}\left(\max\{\mu, \tilde{\mu}\} + \Vert \mathbf{U}\Vert_{L^\infty}\right)^{1/2}$ under assumption \eqref{eq:assumptionH}. Let $v_p$ and $v_{\tilde{p}}$ be the correctors to the cell problem for $p$ and $\tilde{p}$, respectively with $v_p(0) = v_{\tilde{p}}(0) = 0$. We have $\dot{\eta}(s) = D_pH(\eta(s), p+v'_p(\eta(s)))$, therefore, the equality in Fenchel-Young inequality holds:
\begin{align*}
    \int_0^{t} \Big( L(\eta(s), \dot{\eta}(s)) + \overline{H}(\tilde{p})\Big)\;ds &\geq \int_0^t\dot{\eta}(s)\big(\tilde{p} + Dv_{\tilde{p}}(\eta(s))\big)ds = \tilde{p} \eta(t) + v_{\tilde{p}}(\eta(t)),\\
    \int_0^{t} \Big( L(\eta(s), \dot{\eta}(s)) + \overline{H}(p)\Big)\;ds &= \int_0^t \dot{\eta}(s)\big(p + Dv_p(\eta(s))\big)ds = p \eta(t) + v_{p}(\eta(t)).
\end{align*}
Subtracting these equations with $\overline{H}'(p)(\tilde{p} - p)$ we obtain
\begin{equation}\label{eq:estimate-p-tilde-p}
    \overline{H}(\tilde{p}) - \overline{H}(p) - \overline{H}'(p)(\tilde{p} - p) \geq \left(\frac{\eta(t)}{t}-\overline{H}'(p)\right)(\tilde{p}-p) + \frac{v_{\tilde{p}}(\eta(t))}{t} - \frac{v_{p}(\eta(t))}{t}.
\end{equation}
By assumption \eqref{eq:condtion-growth-sublinear} and \eqref{eq:bound-eta}we have
\begin{equation*}
    \left|\frac{v_{p}(\eta(t))}{t}\right| 
    +
    \left|\frac{v_{\tilde{p}}(\eta(t))}{t}\right|  
    \leq 
    \frac{\eta(t)}{t} 
    \left(\frac{(C_p+C_{\tilde{p}})C_0}{\eta(t)^{\theta}}\right) 
    \leq \frac{C}{t^{\theta}}.
\end{equation*}
By assumption that $\overline{H}$ is $\C^{1,\gamma}$ and convexity, we have 
\begin{equation}\label{eq:bound3}
	0 \leq 
	    \overline{H}(\tilde{p}) - \overline{H}(p) 
	    - \overline{H}'(p)(\tilde{p} - p)  
    \leq 
	C|\tilde{p}-p|^{1+\beta}.
\end{equation}
We deduce from \eqref{eq:estimate-p-tilde-p} and \eqref{eq:bound3} that
\begin{equation}\label{eq:bound-master-1}
    \left(\frac{\eta(t)}{t} - \overline{H}'(p)\right)\cdot(\tilde{p}-p) \leq C|\tilde{p} - p|^{1+\beta} + \frac{C}{t^\theta}. 
\end{equation}
At this stage, it is crucial to assume $p,\tilde{p} > p_0$, allowing us to select $\tilde{p} = p + \omega(t)\;\mathrm{sign}\left(\frac{\eta(t)}{t} - \overline{H}'(p)\right)$, where $\omega(t) \to 0^+$ as $t\to +\infty$, to be chosen. Since $|\tilde{p} - p| = \omega(t)$, we obtain from \eqref{eq:bound-master-1} that
\begin{equation*}
    \left|\frac{\eta(t)}{t} - \overline{H}'(p)\right| 
        \leq 
    C\omega(t)^\beta
        + 
        \frac{C}{t^\theta \omega(t)}. 
\end{equation*}
The optimal choice of $\omega(t)$ is $\omega(t)^{1+\beta} = t^{-\theta}$, which means $\omega(t) = t^{-\frac{\theta}{1+\beta}}$ and we obtain the conclusion for $p > p_0$. Now as the constant $C$ in the bound is independent of $p$ as long as $|p_0|\leq |p|\leq M$, and $\overline{H}'(p)$ is continuous as $p\to p_0^+$, we can deduce the same result for $|p|\geq |p_0|$. 
\end{proof}

\section{Rate of convergence for homogenization}
\label{Sect:rate2D}

In this section, we first estimate the rate of convergence for the case $n=2$ with the prototype \ref{itm:A1} or \ref{itm:A2}, which leads to Theorem \ref{thm:rate-C1-H-bar-prototype}. After that, we consider the general case $n\geq 2$ and 
give the proof for Theorem \ref{thm:n-frequency-lower-bound}. The whole process can be divided into four subsections.

\subsection{Ergodic estimates and the growth rate of corrector}

This part is devoted to compute the values of certain norms of $\F_\mu$ defined in \eqref{eq:defn-F-mu}, and reveal its asymptotic behavior as $\mu$ approaches $0^+$. For brevity, we use the notion `$\lesssim$' in the following computations, which implies `$\leq$' by multiplying a $\mathcal{O}(1)$-constant independent of the arguments in the context. For simplicity, we write \(\mathbf{x} = (x_1, x_2) \in \mathbb{T}^2\) to denote a point in \(\mathbb{T}^2\).

\begin{lemma}\label{prop:propertiesV0} Assume \ref{itm:A1}. Let $\mathbf{U}_0(\xi_1,\xi_2) = 2-\sin(2\pi \xi_1) - \sin(2\pi \xi_2)$ for simplicity. We have 
\begin{equation}\label{eq:est-V0}
\begin{cases}
\begin{aligned}
    C_1|\x-\x_0|^2 \leq &\mathbf{U}_0(\x) \leq C_2|\x-\x_0|^2, \quad &|\nabla \mathbf{U}_0(\x)| \leq C|\x-\x_0|, \\
    &|D^2\mathbf{U}_0|\leq C, \quad &|D^3\mathbf{U}_0(\x)| \leq C|\x-\x_0|,
\end{aligned}
\end{cases}
\end{equation}
where $\x_0 = (1/4,1/4)$. For $\mu \geq 0$, let $ \F_\mu = (\mu + \mathbf{U})^{1/2}$ and $\K_\mu = \F_\mu^{-1} = (\mu+\F)^{-{1/2}}$.
\begin{itemize}
    \item[$(\mathrm{i})$] $\F_\mu$ Lipschitz if $\gamma\geq 1$. Furthermore, for $0<\mu \leq C$ and $s\in (0,1)$ we have
    \begin{equation}\label{eq:estimate-H2s-F-mu-no-inverse}
        \begin{tabular}{|c||c|c|c|c|c|} 
        \hline
            $\gamma$            
                & $\gamma < 1$ 
                &  $\gamma = 1$ 
                & $1<\gamma<2$ 
                & $\gamma =2$ 
                & $\gamma > 2$\\
        \hline
            $\Vert \F_\mu\Vert_{H^{2+s}(\T^2)}$ 
                & $C \mu^{\frac{1}{2}-\frac{1+s}{2\gamma}}$
                & $C |\log(\mu)|^{\frac{1-s}{2}}\mu^{-\frac{s}{2}}$
                & $C \mu^{\frac{s}{2}-\frac{s}{\gamma}}$
                & $C |\log(\mu)|^{\frac{s}{2}}$
                & $C$ \\
        \hline
        \end{tabular}
    \end{equation}
    \noindent

    \item[$(\mathrm{ii})$] The function $\F_\mu^{-1} = \K_\mu$ satisfies
    \begin{equation*}
    \Vert \K_\mu\Vert_{L^1(\T^2)}
    \lesssim 
    \begin{cases}
        C(1+|\log(\mu)|) &\qquad\text{if}\;\gamma = 2,\\    
        C\big(1+ \mu ^{\frac{1}{\gamma} - \frac{1}{2}}\big) &\qquad\text{if}\;\gamma \neq 2,.
    \end{cases}
    \end{equation*}

    \item[$(\mathrm{iii})$] For each $s\in (0,1)$ there exists $C=C(\gamma,s)$ such that
    \begin{align}\label{eq:estimate-H2s-INVERSE}
        \Vert \K_\mu\Vert_{H^{2+s}(\T^2)} \leq
           C\left(1+\mu^{-\frac{1}{2} - \frac{1+s}{2\gamma}}\right), \qquad \gamma \in (0,\infty).
    \end{align}
\end{itemize}
\end{lemma}

\begin{proof} Let $\x_0 = (1/4, 1/4)$, from \eqref{eq:est-V0} and direct computation we have
\begin{equation}\label{eq:DF}
    \nabla \F_\mu = \frac{1}{2}(\mu+\mathbf{U}_0^\gamma)^{-1/2} \gamma \mathbf{U}_0^{\gamma-1} \nabla \mathbf{U}_0 = \frac{\gamma}{2}
    \frac{\mathbf{U}_0^{\gamma-1} \nabla \mathbf{U}_0}{\left(\mu + \mathbf{U}_0^\gamma\right)^{1/2}} \leq \frac{C|\x-\x_0|^{2\gamma-1}}{|\x-\x_0|^\gamma} \leq C_2 |\x-\x_0|^{\gamma-1}.
\end{equation}
\paragraph{\underline{\textbf{Part (i)}}} 
From \eqref{eq:DF} we have $|\nabla \F_\mu|\leq C$ if $\gamma\geq 1$. To estimate the Sobolev norm $\Vert \F_\mu\Vert_{H^{2+s}(\T^2)}$ by computing $\Vert D^k\F\Vert_{L^2(\T^2)}$ for $k=0,1,2,3$. Thanks to \eqref{eq:est-V0}, we compute directly that 
\begin{align}
    &\Vert D^0\F_\mu\Vert_{L^2(\mathbb{T}^2)} 
    = \Vert \F_\mu\Vert_{L^2(\mathbb{T}^2)} \leq C(1+\mu)^{1/2} 
    \nonumber, \\ 
    &\Vert D\F_\mu\Vert_{L^2(\mathbb{T}^2)} 
    \lesssim  
    C\left(\int_{\T^2} |\x-\x_0|^{2(\gamma-1)}\;d\x\right)^{1/2} \leq C\left(1 + \int_0^1 r^{2(\gamma-1)+1}\;dr    \right) \leq C
    \nonumber, \\
     &\Vert D^2\F_\mu\Vert_{L^2(\mathbb{T}^2)}  \lesssim 
    \begin{cases}
            C\left(1 + |\log(\mu)|^{\frac{1}{2}}\right), 
                    &\qquad \text{if}\;\gamma = 1,\\
            C\left(1 + \mu^{\frac{1}{2}\left(1-\frac{1}{\gamma}\right)}\right),
                    &\qquad \text{if}\;\gamma \neq 1. 
    \end{cases} 
    \nonumber
    \\
    &\Vert D^3\F_\mu\Vert_{L^2(\mathbb{T}^2)}  \lesssim 
    \begin{cases}
            C\left(1 + |\log(\mu)|^{\frac{1}{2}}\right), 
                    &\qquad \text{if}\;\gamma = 2,\\
            C\left(1 + \mu^{\frac{1}{2}\left(1-\frac{2}{\gamma}\right)}\right),
                    &\qquad \text{if}\;\gamma \neq 2.
    \label{eq:D3Fmu}
    \end{cases}
\end{align}
Using $\Vert \F_\mu\Vert_{H^{2}(\T^2)} 
    \lesssim \Vert \F_\mu\Vert_{L^2(\mathbb{T}^2)}  + \Vert D\F_\mu\Vert_{L^2(\mathbb{T}^2)}+
    \Vert D^2\F_\mu\Vert_{L^2(\mathbb{T}^2)}$, we deduce that
\begin{equation}\label{eq:AsymptoticOfH2sNormFmu}
    \begin{tabular}{|c||c|c|c|c|} 
        \hline
            $\gamma$            
                & $\gamma < 1$ 
                &  $\gamma = 1$ 
                & $\gamma > 1$\\
        \hline
            $\Vert \F_\mu\Vert_{H^{2}(\T^2)}$ 
                & $C_\mu \mu^{\frac{1}{2}\left(1-\frac{1}{\gamma}\right)}$
                & $C_\mu |\log(\mu)|^\frac{1}{2}$
                & $C_\mu$ \\
        \hline
    \end{tabular}
\end{equation}
Using Lemma \ref{prop:Holder-Sobolev}, for $s\in (0,1)$ we obtain
\begin{align}\label{eq:interpolationA1}
	\Vert \F_\mu\Vert_{H^{2+s}(\T^2)}  
		&\leq 	\Vert \F_\mu\Vert_{H^2(\T^2)} ^{1-s} 	\Vert \F_\mu\Vert_{H^3(\T^2)} ^s \leq 	\Vert \F_\mu\Vert_{H^{2+s}(\T^2)} + \Vert \F_\mu\Vert_{H^{2+s}(\T^2)}^{1-s}\Vert D^3\F_\mu\Vert_{L^{2}(\T^2)}^s.
\end{align}
From \eqref{eq:D3Fmu}, \eqref{eq:AsymptoticOfH2sNormFmu} and \eqref{eq:interpolationA1}
we obtain the conclusion \eqref{eq:estimate-H2s-F-mu-no-inverse}. 
    \medskip

\paragraph{\underline{\textbf{Part (ii)}}} We compute using \eqref{eq:est-V0} that
\begin{align*}
    \int_{\mathbb{T}^2} \K_\mu(\x)\;d\x 
    &\lesssim C\left(1 + \int_{|\x-\x_0|<1} \frac{d\x}{(\mu + |\x-\x_0|^{2\gamma})^{1/2}}\right) \lesssim C\left(1 + \int_{0}^{1}\frac{r\;dr}{(\mu+r^{2\gamma})^{1/2}}\right) \\
    &\lesssim C\left(1+\mu^{\frac{1}{\gamma} - \frac{1}{2}}\int_0^{\mu^{-\frac{1}{2\gamma}}}\frac{s\;ds}{(1+s^{2\gamma})^{1/2}}\right) 
    \lesssim 
    \begin{cases}
        C(1+|\log(\mu)|) &\qquad\text{if}\;\gamma = 2,\\    
        C\left(1+ \mu ^{\frac{1}{\gamma} - \frac{1}{2}}\right) &\qquad\text{if}\;\gamma \neq 2. 
    \end{cases}
\end{align*}
The conclusion of $\Vert \K_\mu \Vert _{L^1(\T^2)}$ follows from there. \medskip

\paragraph{\underline{\textbf{Part (iii)}}}
We compute using \eqref{eq:est-V0} that 
\begin{align*}
    \Vert \K_\mu\Vert_{L^2(\T^2)} 
    &\lesssim C\left(1 + \int_{|\x-\x_0|<1} \frac{d\x}{(\mu + |\x-\x_0|^{2\gamma})}\right)^{1/2} \lesssim C\left(1 + \int_{0}^{1}\frac{r\;dr}{(\mu+r^{2\gamma})}\right)^{1/2} \\
    &\lesssim C\left(1+\mu^{\frac{1}{\gamma} - 1}\int_0^{\mu^{-\frac{1}{2\gamma}}}\frac{s\;ds}{(1+s^{2\gamma})}\right)^{1/2} \lesssim 
    \begin{cases}
        C\left(1+|\log(\mu)|^{\frac{1}{2}}\right)                & \gamma=1,\\
        C\left(1+\mu^{-\frac{1}{2}+\frac{1}{2\gamma}}\right)  & \gamma\neq 1.
    \end{cases}
\end{align*}
Computing $D\K_\mu, D^2\K_\mu$ and $D^3\K_\mu$ and using the same method as in the previous parts, we have
\begin{align*}
	\Vert D\K_\mu\Vert_{L^2(\T^2)} \lesssim C\left(1+\mu^{-\frac{1}{2}}\right), \quad 
	\Vert D^2\K_\mu\Vert_{L^2(\T^2)} \lesssim C\left(1+\mu^{-\frac{1}{2} - \frac{1}{2\gamma}}\right), \quad 
	 \Vert D^3\K_\mu\Vert_{L^2(\T^2)} \lesssim C\left(1+\mu^{-\frac{1}{2} - \frac{1}{\gamma}}\right).
\end{align*}
We obtain that $\Vert \K_\mu\Vert_{H^{2}(\T^2)} 
    \lesssim \Vert \K_\mu\Vert_{L^2(\mathbb{T}^2)}  + \Vert D\K_\mu\Vert_{L^2(\mathbb{T}^2)}+
    \Vert D^2\K_\mu\Vert_{L^2(\mathbb{T}^2)} \leq C\mu^{-\frac{1}{2}-\frac{1}{2\gamma}}$. Using Lemma \ref{prop:Holder-Sobolev} similarly to \eqref{eq:interpolationA1} we obtain \eqref{eq:estimate-H2s-INVERSE}.
\end{proof}

Similar conclusions to those in Lemma \ref{prop:propertiesV0} can be drawn under assumption \ref{itm:A2}, with necessary adaptations. In the following, we treat \ref{itm:A1} and \ref{itm:A2} equally and state conclusions for both.

\begin{corollary}\label{coro:rate-v0} Assume \ref{itm:A1} or \ref{itm:A2} and $\sigma_\xi = 1$. If $0\leq \mu\leq \mu_0$ then there exists $C=C(\gamma, \mu_0)$ such that 
\begin{numcases}
    {\left|
        \frac{1}{T}\int_0^T (\mu+\mathbf{U}(\xi x))^{1/2}\;d x  - \int_{\T^2} (\mu+\mathbf{U}(\x))^{1/2}\;d\x
    \right|
    \leq }
        CT^{-1} & $\gamma \in (2,\infty)$,
        \nonumber\\ 
        CT^{-\frac{1}{2}} & $\gamma \in [1,2]$,
        \label{eq:LowerRateA1GammaBet1And2}\\ 
        CT^{-\frac{\gamma}{\gamma+1}} & $\gamma \in (0,1)$.
        \label{eq:LowerRateA1GammaLess1}
\end{numcases}
Moreover, for $\gamma > 2$, we can relax the condition $\sigma_\xi = 1$ to
\begin{equation}\label{eq:LowerRateA1GammaGeneralizedXi2}
    \left|\frac{1}{T} \int_0^T (\mu + \mathbf{U}(\xi x))^{1/2}\;dx - \int_{\T^2} (\mu + \mathbf{U}(\x))^{1/2} d\x\right| \leq \frac{C}{T} \qquad\text{for all}\; \xi\in \R^2\;\text{with}\;\sigma_\xi < 2. 
\end{equation}
\end{corollary}
\begin{proof} The proof follows from the calculation of the $\Vert \mathbf{U}^{1/2}\Vert_{H^{2+s}(\T^2)}$ or the H\"older norm of $\mathbf{U}^{1/2}$, which we carry out in Lemma \ref{prop:propertiesV0}. If $\gamma > 2$ then by \eqref{eq:estimate-H2s-F-mu-no-inverse} in Lemma \ref{prop:propertiesV0}, we have $\mathbf{U}^{1/2} \in H^{2+s}(\T^2)$ for any $s \in [0,1]$. As $n=2$, we have $\sigma_\xi \geq n-1 = 1$ (see \cite{JurgenKAM1999}). Since $\sigma_\xi \in [1,2)$, we can choose $s_0 \in (\sigma_\xi-1, 1)$, then with this value of $s_0$ we have $\mathbf{U}^{1/2} \in H^{2+s_0}(\T^2)$, and $2+s_0 > \frac{n}{2}+\sigma_\xi$. 
We can apply Proposition \ref{prop:ap:rate-Mean-QR} to obtain \eqref{eq:LowerRateA1GammaGeneralizedXi2}. If $\gamma\geq 1$ then by Lemma \ref{prop:propertiesV0} (i) we have $\mathbf{U}^{1/2}\in \C^{0,1}(\T^2)$, while if $\gamma\leq 1$, we have $\mathbf{U}^{1/2}\in \C^{0,\gamma}(\T^2)$. By Proposition \ref{prop:n-frequency-rate} we obtain \eqref{eq:LowerRateA1GammaBet1And2} and \eqref{eq:LowerRateA1GammaLess1}, respectively.
\end{proof}

\begin{remark}\label{remark:non-optimal-rate} We observe a discrepancy in the convergence rate in Corollary \ref{coro:rate-v0} between $\gamma > 2$ and $\gamma\leq 2$. This gap arises from the differences in the methods used to estimate the convergence rate to the mean value in Propositions \ref{prop:ap:rate-Mean-QR} and \ref{prop:n-frequency-rate}. 
\end{remark}


\begin{corollary}\label{coro:growth-vp-underA1} Assume \ref{itm:A1} or \ref{itm:A2} and $\sigma_\xi = 1$. If $|p_0|\leq p \leq M$ and $v_p$ is a corrector to \eqref{eq:CPdelta}, then there exists a positive constant $C = C(M)$ such that
\begin{numcases}
{\left|\frac{v_p(t)}{t}\right| 
        		\leq 
        	\frac{C}{|t|^\theta} 	
        		\qquad
        	\text{where}
        		\qquad
        	\theta = }
            \textstyle
        	 \;\;  1 
            &\text{if}\;$\gamma \in (2,\infty)$, 
            \label{eq:correctorBig2}\\
            \textstyle
        \;\; \frac{1}{2}
            &\text{if}\;$\gamma \in [1,2]$,
            \nonumber\\
            \textstyle
        \frac{\gamma}{\gamma+1}
            &\text{if}\;$\gamma \in (0,1)$. 
            \label{eq:correctorLess2}
\end{numcases}
    Furthermore, when $\gamma>2$, \eqref{eq:correctorBig2} is true for all $\xi\in \R^2$ such that $\sigma_\xi < 2$. 
\end{corollary}
\begin{proof} Using \eqref{eq:GrowthRateCorrector} in Corollary \ref{prop:rate-corrector-sublinear-badly-approximable} together with Corollary \ref{coro:rate-v0} we obtain the conclusion.
\end{proof}

\subsection{Regularity of $\overline{H}$ and large time average of characteristics}

Assume \ref{itm:A1} or \ref{itm:A2}, we aim to establish certain $\C^{1,\beta}$ regularity of $\overline{H}$, then use it to derive a rate of convergence for the large time average of characteristics in \eqref{eq:rate:eta-average-full}. Recall that $\overline{H}\in \C^1(\R\backslash [-p_0, p_0])$ from Lemma \ref{prop:D1H-bar-derivative}. 
In what follows, we denote $\overline{H}'(p) = \overline{H}'_+(p_0)$ if $p = p_0$. 

\begin{lemma}\label{lem:BehaviorDerevativeEffectiveHbar} Assume \ref{itm:A1} or \ref{itm:A2}. 
\begin{itemize}
    \item[$\mathrm{(i)}$] $\overline{H}\in \mathrm{C}^1(\mathbb{R})$ if and only if $\gamma\geq 2$. Furthermore, if $p\geq p_0$ there exist some $C_1, C_2 > 0$ such that
    \begin{align}
        C_1\overline{H}(p)^{\frac{1}{2} - \frac{1}{\gamma}}
            &\leq 
        \overline{H}'(p)
            \leq 
        C_2\overline{H}(p)^{\frac{1}{2} - \frac{1}{\gamma}}  &&  \gamma \in (2,\infty),
        \label{eq:bound-2-sides-DH-bar-gamma-not-2} \\
        C_1|\log(\overline{H}(p))|^{-1} 
            &\leq 
        \overline{H}'(p)
            \leq 
        C_2|\log(\overline{H}(p))|^{-1}  &&  \gamma = 2, \label{eq:bound-2-sides-DH-bar-gamma-2}\\
        0< C_1
            &\leq 
        \overline{H}'(p)
            \leq 
        C_2  && \gamma \in (0,2).
        \label{eq:bound-2-sides-DH-bar-gamma-less-2}
    \end{align}
    \item[$\mathrm{(ii)}$]  $\overline{H}\notin\mathrm{C}^2(\mathbb{R})$ for any $\gamma>0$. Furthermore, if $p > p_0$ there exist some $C_1, C_2 > 0$ such that 
    \begin{align}
        C_1\overline{H}(p)^{-\frac{2}{\gamma}}
            &\leq 
        \overline{H}''(p)
            \leq 
        C_2\overline{H}(p)^{- \frac{2}{\gamma}}  &&  \gamma \in (2,\infty)
        \label{eq:bound-2-sides-D2H-bar-gamma-bigger-2-and-2/3} \\
        C_1\left(\overline{H}(p)\log(\overline{H}(p))|\right)^{-1}
            &\leq 
        \overline{H}''(p)
            \leq 
        C_2\left(\overline{H}(p)\log(\overline{H}(p))|\right)^{-1}  &&  \gamma =2 
        \label{eq:bound-2-sides-D2H-bar-gamma-=2} \\ 
        C_1\overline{H}(p)^{\frac{1}{\gamma}-\frac{3}{2}}
            &\leq 
        \overline{H}''(p)
            \leq 
        C_2\overline{H}(p)^{\frac{1}{\gamma}-\frac{3}{2}}  &&   \gamma \in (\textstyle\frac{2}{3},2),
        \label{eq:bound-2-sides-D2H-bar-gamma-less-2} \\ 
        C_1(1+ |\log(\overline{H}(p)|)
            &\leq 
        \overline{H}''(p)
            \leq 
        C_2(1+ |\log(\overline{H}(p)|)  &&  \gamma=\textstyle\frac{2}{3}
        \label{eq:bound-2-sides-D2H-bar-gamma-AT23} \\ 
        C_1
            &\leq 
        \overline{H}''(p)
            \leq 
        0< C_2  &&  \gamma \in (0,\textstyle\frac{2}{3}).
        \label{eq:bound-2-sides-D2H-bar-gamma-LESS23} 
    \end{align}
\end{itemize}

\end{lemma}

\begin{proof} For $\mu > 0$, let $\mu\mapsto \phi(\mu) = p$ be the inverse of $p\mapsto \overline{H}(p)$ with $\phi(\mu)$ and $\phi'(\mu)$ defined in \eqref{eq:defn-F-mu} and \eqref{eq:def-derivtive-Df-inverse-H-bar}, respectively. \medskip

\paragraph{\underline{\textbf{Part (i)}}} By direct computation, $\nabla \F_\mu$ is bounded if $\gamma\geq 1$, and there are $C_1,C_2>0$ such that
\begin{equation*}
\begin{aligned}
    C_1(1+|\log(\mu)|) \leq \; &\phi'(\mu) = \int_{\mathbb{T}^2} \frac{d\x}{(\mu+\mathbf{U}(\x)^{1/2}} \leq C_2(1+|\log(\mu)|) & &\qquad \gamma = 2, \\
    C_1(1+\mu^{\frac{1}{\gamma}-\frac{1}{2}}) \leq \; &\phi'(\mu) = \int_{\mathbb{T}^2} \frac{d\x}{(\mu+\mathbf{U}(\x)^{1/2}} \leq C_2(1+\mu^{\frac{1}{\gamma}-\frac{1}{2}}) & &\qquad \gamma \neq 2.
\end{aligned}
\end{equation*}
Since $\phi'(\mu) = \overline{H}'(p)^{-1}$, we obtain the conclusions. \medskip

\paragraph{\underline{\textbf{Part (ii)}}} To compute $\overline{H}''(p)$ for $|p|>p_0$, we use $\overline{H}(\phi(\mu)) = \mu$ to obtain
\begin{equation*}
    \overline{H}''(p) 
    = -\phi''(\mu) \Big(\overline{H}'(p)\Big)^3 
    = \Big(\overline{H}'(p)\Big)^3 \int_{\mathbb{T}^2}
    \frac{d\x}{\left(2(\mu + \mathbf{U}(\x))\right)^{3/2}},
\end{equation*}
where we compute $\phi''(\mu)$ from \eqref{eq:def-derivtive-Df-inverse-H-bar}. We have 
\begin{align}
    C_1 \left(1 + \mu^{\frac{1}{\gamma}-\frac{3}{2}}\right)
    \leq 
    & \int_{\mathbb{T}^2}
    \frac{d\x}{\left(2(\mu + \mathbf{U}(\x))\right)^{3/2}} 
    \leq 
    C_2\left(1 + \mu^{\frac{1}{\gamma} - \frac{3}{2}}\right) 
    &\qquad  \gamma \neq \frac{2}{3}, 
    \nonumber \\ 
    C_1 \left(1 + |\log(\mu)|\right)
    \leq 
    &\int_{\mathbb{T}^2}
    \frac{d\x}{\left(2(\mu + \mathbf{U}(\x))\right)^{3/2}} 
    \leq 
    C_2\left(1 + |\log(\mu)|\right) 
    &\qquad  \gamma = \frac{2}{3}.  
    \nonumber
\end{align}
Therefore we obtain the conclusions from the previous estimates \eqref{eq:bound-2-sides-DH-bar-gamma-not-2}, \eqref{eq:bound-2-sides-DH-bar-gamma-2}, and \eqref{eq:bound-2-sides-DH-bar-gamma-less-2}. 
\end{proof}

\begin{lemma}\label{lem:BehaviorHolderDerevativeEffectiveHbar} Assume \ref{itm:A1} or \ref{itm:A2}. For $\tilde{p}, p\geq p_0$ we have
    \begin{align*}
        |\overline{H}'(\tilde{p}) - \overline{H}'(p)| &\leq C|\tilde{p} - p|^{\frac{1}{2} - \frac{1}{\gamma}} 
             \qquad && \gamma \in (2,\infty)\\
        |\overline{H}'(\tilde{p}) - \overline{H}'(p)| &\leq C\left|\log\left(\overline{H}(\tilde{p})\right) - \log\left(\overline{H}(p)\right)\right| 
             \qquad &&\gamma = 2\\
        |\overline{H}'(\tilde{p}) - \overline{H}'(p)| &\leq C|\tilde{p} - p|^{\frac{1}{\gamma}-\frac{1}{2}} 
             \qquad && \gamma \in (\textstyle\frac{2}{3}, 2) \\
        |\overline{H}'(\tilde{p}) - \overline{H}'(p)| &\leq C 
        \left| 
            \overline{H}(\tilde{p})(1+|\log(\overline{H}(\tilde{p}))|) 
            - 
            \overline{H}(p)(1+|\log(\overline{H}(p))|) 
        \right| 
             \qquad && \gamma = \textstyle\frac{2}{3}\\ 
        |\overline{H}'(\tilde{p}) - \overline{H}'(p)| &\leq C|\tilde{p} - p| 
             \qquad && \gamma \in (0, \textstyle\frac{2}{3}). 
    \end{align*}
\end{lemma}
\begin{proof} Let $\tilde{p}, p \geq p_0$. For brevity, we denote $\tilde{\mu} = \overline{H}(\tilde{p})$ and $\mu = \overline{H}(p)$. We have 
\begin{align}\label{eq:est-for-H'-Holder}
    \overline{H}'(\tilde{p}) - \overline{H}'(p) 
        &\leq \int_{p}^{\tilde{p}} \overline{H}''(\xi)\;d\xi.
\end{align}
We use the change of variable $\xi = \phi(\nu)$ and $\nu = \overline{H}(\xi)$ from \eqref{eq:defn-F-mu}, recall that $\phi'(\nu) = \overline{H}'(\xi)^{-1}$. 
\begin{itemize}
    \item If $\gamma > 2$, we use \eqref{eq:bound-2-sides-DH-bar-gamma-not-2} and \eqref{eq:bound-2-sides-D2H-bar-gamma-bigger-2-and-2/3}, which correspond to $\phi'(\nu) \leq C\nu^{\frac{1}{\gamma} - \frac{1}{2}}$ and $\overline{H}''(\xi) \leq C\nu^{-\frac{2}{\gamma}}$ to obtain 
    \begin{align*}
        \int_{p}^{\tilde{p}} \overline{H}''(\xi)\;d\xi 
        &\leq C\int_p^{\tilde{p}}\overline{H}(\xi)^{-\frac{2}{\gamma}}\;d\xi = C\int_{\mu}^{\tilde{\mu}} \nu^{-\frac{2}{\gamma}} \phi'(\nu)\;d\nu \leq C\int_{\mu}^{\tilde{\mu}} \nu^{-\frac{1}{2} - \frac{1}{\gamma}}\;d\nu\\
        &\leq C\left(\tilde{\mu}^{\frac{1}{2}-\frac{1}{\gamma}} - \mu^{\frac{1}{2}-\frac{1}{\gamma}} \right) 
        \leq C|\tilde{\mu}-\mu|^{\frac{1}{2}-\frac{1}{\gamma}} = C|\overline{H}(\tilde{p}) - \overline{H}(p)|^{\frac{1}{2}-\frac{1}{\gamma}} \leq C|\tilde{p} - p|^{\frac{1}{2}-\frac{1}{\gamma}}.
    \end{align*}
    \item  If $\gamma=2$, we use \eqref{eq:bound-2-sides-DH-bar-gamma-2} and \eqref{eq:bound-2-sides-D2H-bar-gamma-=2}, which are $\phi'(\nu)\leq C|\log(\nu)|$ and $\overline{H}''(\xi)\leq C\nu^{-1}|\log(\nu)|^{-1}$ to obtain 
    \begin{align*}
        \int_{p}^{\tilde{p}} \overline{H}''(\xi)\;d\xi 
        &\leq 
        C\int_p^{\tilde{p}}\frac{d\xi}{\nu|\log(\nu)|} 
        = C\int_{\mu}^{\tilde{\mu}} 
        		\frac{\phi'(\nu)\;d\nu}{\nu|\log(\nu)|} 
        \leq C\int_{\mu}^{\tilde{\mu}} \frac{d\nu}{\nu} 
        = C\left(\log(\tilde{\mu}) - \log(\mu)\right).
    \end{align*}
    \item If $\frac{2}{3}< \gamma < 2$, we use \eqref{eq:bound-2-sides-DH-bar-gamma-less-2} and \eqref{eq:bound-2-sides-D2H-bar-gamma-less-2}, which are $C_1 \leq \phi'(\nu) \leq C_2$ and $\overline{H}''(\xi) \leq C\nu^{\frac{1}{\gamma} - \frac{3}{2}}$ to obtain 
    \begin{align*}
        \int_{p}^{\tilde{p}} \overline{H}''(\xi)\;d\xi 
        &\leq 
        C\int_{\mu}^{\tilde{\mu}} 
        \nu^{\frac{1}{\gamma} - \frac{3}{2}}\phi'(\nu)\;d\nu 
        \leq C \int_{\mu}^{\tilde{\mu}} \nu^{\frac{1}{\gamma} - \frac{3}{2}} \;d\nu  
        = C
        \left(
        	\tilde{\mu}^{\frac{1}{\gamma}-\frac{1}{2}} 
        		- 
        	\mu^{\frac{1}{\gamma}-\frac{1}{2}}
        \right) 
        \leq C|\tilde{p} - p|^{\frac{1}{\gamma}-\frac{1}{2}}. 
    \end{align*}
    \item If $\gamma=\frac{2}{3}$, we use \eqref{eq:bound-2-sides-DH-bar-gamma-less-2} and \eqref{eq:bound-2-sides-D2H-bar-gamma-AT23}, which are $C_1 \leq \phi'(\nu) \leq C_2$ and $\overline{H}''(\xi) \leq C|\log(\nu)|$ to obtain 
    \begin{align*}
        \int_{p}^{\tilde{p}} \overline{H}''(\xi)\;d\xi \leq C\int_{\mu}^{\tilde{\mu}} |\log(\nu)|\;d\nu = C\Big(\tilde{\mu}(1+|\log(\tilde{\mu})|) - \mu(1+|\log(\mu)|)\Big).
    \end{align*}
    \item If $0<\gamma<\frac{2}{3}$, we use \eqref{eq:bound-2-sides-DH-bar-gamma-less-2} and \eqref{eq:bound-2-sides-D2H-bar-gamma-LESS23}, which are $C_1 \leq \phi'(\nu) \leq C_2$ and $\overline{H}''(\xi) \leq C$ to obtain
    \begin{align*}
        \int_{p}^{\tilde{p}} \overline{H}''(\xi)\;d\xi \leq C\Big(\overline{H}(\tilde{p}) - \overline{H}(p)\Big) \leq C|\tilde{p} - p|.
    \end{align*}
\end{itemize}
We deduce the conclusion from \eqref{eq:est-for-H'-Holder}.
\end{proof}

\begin{corollary}\label{prop:gamma=2-log-est} Assume \ref{itm:A1} or \ref{itm:A2}. We have
\begin{equation}\label{eq:HolderHBar}
	\begin{cases}
		\overline{H}\in \C^{1,\frac{1}{2}-\frac{1}{\gamma}}(\R) 
			& \qquad \gamma \in (2,\infty), \\
		\overline{H}\in \C^{1,\frac{1}{\gamma}-\frac{1}{2}}(\R\backslash [-p_0, p_0]) 
			& \qquad \gamma \in \left(\frac{2}{3}, 2\right), \\
		\overline{H}\in \C^{1,1}(\R\backslash [-p_0, p_0]) 
			& \qquad \gamma \in \left(0,\frac{2}{3}\right). 
	\end{cases}
\end{equation}
\end{corollary}

\begin{proposition}\label{coro:average-A1} Assume \ref{itm:A1} or \ref{itm:A2} and $\sigma_\xi = 1$. For $|p|\geq |p_0|$, let  $v_p$ be an exact sublinear corrector to \eqref{eq:CPdelta} and $\eta:[0,\infty) \to \mathbb{R}$ be a characteristic with respect to $p$ with $\eta(0) = 0$, or equivalently, $\eta$ is a solution of 
\begin{equation*}
\begin{cases}
\begin{aligned}
    |\dot{\eta}(s)| 
    		&= \sqrt{2(\overline{H}(p) - V(\eta(s))}, \qquad s>0,\\
    \eta(0) 
    		&= 0.
\end{aligned}
\end{cases}
\end{equation*}
If $|p_0|\leq |p|\leq M$, then there exists a positive constant $C=C(M)$ such that
\begin{align}
    \left|\frac{\eta(t)}{t} - \overline{H}'(p)\right|
        &\leq \frac{C}{|t|^\tau}\qquad\text{where}\;\tau =  \frac{\gamma-2}{3\gamma-2} 
            		\in \left(0,\frac{1}{3}\right) &&\gamma \in (2,\infty) 
              \label{eq:polyrateforgammaneq2}\\
    \left|\frac{\eta(t)}{t} - \overline{H}'(p)\right| 
        &\leq 
        \frac{C}{|\log(t)|} && \gamma = 2
        \nonumber\\
    \left|\frac{\eta(t)}{t} - \overline{H}_+'(p)\right|
        &\leq \frac{C}{|t|^\tau}\qquad\text{where}\;\tau = \frac{2-\gamma}{2(2+\gamma)} 
            		\in \left(0,\frac{1}{6}\right] && \gamma \in [1,2). 
              \nonumber
\end{align}
\end{proposition}

\begin{proof} We can assume $p>p_0$ first, then as the constant $C_M$ is bounded as $p\to p_0^+$, we deduce the result for $p\geq p_0$, and the other case $p<-p_0$ is similar. Using Lemma \ref{prop:gamma=2-log-est} and Corollary \ref{coro:growth-vp-underA1}, we proceed as follows. \medskip

\paragraph{\underline{\textbf{Case 1}}} If $\gamma >2$, then we can apply Proposition \ref{prop:boostrap-rate} with $\alpha = \frac{1}{2} - \frac{1}{\gamma}$ and $\theta = 1$, thus 
\begin{equation*}
	 \left|\frac{\eta(t)}{t} - \overline{H}'(p)\right| 
	 \leq 
	 	C\left(\frac{1}{|t|}\right)^\tau 
	 \qquad\text{where}\;
	 \tau = \frac{\theta \alpha}{1+\alpha} 
	       = \frac{\gamma-2}{3\gamma-2} . 
\end{equation*}
\paragraph{\underline{\textbf{Case 2}}} If $\gamma \in [1,2)$ then we can apply Proposition \ref{prop:boostrap-rate} with $\alpha = \frac{1}{\gamma} - \frac{1}{2}$ and $\theta = \frac{1}{2}$, thus 
\begin{equation*}
	\left|\frac{\eta(t)}{t} - \overline{H}_+'(p)\right| 
	 \leq 
	 	C\left(\frac{1}{|t|}\right)^\tau
	 \qquad\text{where}\;
	 \tau = \frac{\theta \alpha}{1+\alpha} 
	       = \frac{(2-\gamma)}{2(2+\gamma)}. 
\end{equation*}

\paragraph{\underline{\textbf{Case 3}}} If $\gamma=2$, let $p, \tilde{p}>p_0$ and $\mu=\overline{H}(p), \tilde{\mu} = \overline{H}(\tilde{p}) > 0$, respectively. Similar to Proposition \ref{prop:boostrap-rate} we have
\begin{align}
	\left(\frac{\eta(t)}{t} - \overline{H}'(p)\right)
	\cdot
	(\tilde{p} - p) 
		&\leq 
	\Big|
            \overline{H}'(\tilde{p}) - \overline{H}'(p)
        \Big|\cdot
	|\tilde{p} - p|  + \frac{C}{|t|^{1/2}} .
        \label{eq:casegamma2log}  
\end{align}
Let 
\begin{equation*}
    \tilde{p} = p + \frac{1}{C_H}\left(\frac{1}{ |t|} \right)^{1/6} \mathrm{sign}\left(\frac{\eta(t)}{t} - \overline{H}'(p)\right).
\end{equation*}
We note that the choice of $p, \tilde{p}$ as in the previous equation is always possible if $p,\tilde{p} > p_0$. 
Since $\overline{H}$ is Lipschitz, we have $|\tilde{\mu} - \mu| \leq C_H|\tilde{p}-p| = |t|^{-1/6}$. There are two cases:
\begin{itemize}
    \item If $\mu > |t|^{-1/3}$, we use the case $\gamma=2$ in Lemma \ref{prop:gamma=2-log-est} to obtain 
    \begin{align*}
    \Big|
        \overline{H}'(\tilde{p}) - \overline{H}'(p) 
    \Big|
        \leq 
    \left|\log\left(1+\frac{\tilde{\mu} - \mu}{\mu}\right)\right|    \leq 
    \frac{\left|\tilde{\mu} - \mu\right|}{\mu} 
       \leq \frac{|t|^{-1/6}}{|t|^{-1/3}}  = \frac{1}{|t|^{1/6}}. 
    \end{align*}    
    We then obtain from \eqref{eq:casegamma2log} that
    \begin{equation*}
        \left|\frac{\eta(t)}{t} - \overline{H}'(p)\right| 
            \leq 
        \frac{1}{|t|^{1/6}} + \frac{C}{|t|^{1/2} \cdot |\tilde{p} - p|} = \frac{1}{|t|^{1/6}} + \frac{C}{|t|^{1/3}} \leq \frac{C}{|t|^{1/6}}. 
    \end{equation*}
    \item If $\mu \leq |t|^{-1/3}$, for $|t|$ large we have
    \begin{equation*}
        \tilde{\mu} \leq \mu + |t|^{-1/6} \leq |t|^{-1/3} + |t|^{-1/6} \leq 2|t|^{-1/6}  .
    \end{equation*}
    Since $x\mapsto |\log(x)|$ is decreasing and $|\log(2x)|\geq \frac{1}{2}|\log(x)|$ if $0<x<\frac{1}{4}$, we have
    \begin{equation*}
        |\log(\mu)| \geq C|\log(t)| \qquad\text{and} \qquad |\log(\tilde{\mu})| \geq C|\log(t)|. 
    \end{equation*}
    From \eqref{eq:casegamma2log} and $\overline{H}'(p) \leq C|\log(\mu)|^{-1}$ (Lemma \ref{prop:gamma=2-log-est}), we have 
    \begin{equation*}
    \begin{aligned}
        \left|\frac{\eta(t)}{t} - \overline{H}'(p)\right| 
            \leq 
        C
        \left(\frac{1}{|\log(\tilde{\mu})|} + \frac{1}{|\log(\mu)|}\right) +\frac{C}{|t|^{1/2}|\tilde{p}-p| } \leq  \frac{C}{|\log(t)|} + \frac{C}{|t|^{1/3}} \leq \frac{C}{|\log(t)|}.
    \end{aligned}
    \end{equation*}
   \end{itemize}
    Combining the two cases we obtain the conclusion for $\gamma=2$. 
\end{proof}

\subsection{Rate of convergence in homogenization}
\label{SubSection:rate2D}

In this section we derive the rate of convergence for $u^\varepsilon\to u$, the solutions to \eqref{eq:intro:Ceps} and \eqref{eq:intro:C}, respectively. By optimal control theory (see \cite{Bardi1997,le_dynamical_2017, tran_hamilton-jacobi_2021}), we can write the solution to \eqref{eq:intro:Ceps} as (the Lagrangian \( L \) is defined as given in \eqref{eq:LegendreTransform})
\begin{align}
     u^\varepsilon(x,t) 
     &= \inf 
     \left\lbrace 
     \int_0^t L\left(\frac{\gamma(s)}{\varepsilon},\dot{\gamma}(s)\right)\;ds + u_0(\gamma(0)) : \gamma(t) = x, \dot{\gamma}\in L^1([0,t])
     \right\rbrace \nonumber \\
     &= \inf 
    \left\lbrace 
    \varepsilon\int_0^{\varepsilon^{-1}} 
    L\left(\eta(s),-\dot{\eta}(s)\right)\;ds
     + 
     u_0(\varepsilon\eta(\varepsilon^{-1})):
    \varepsilon\eta(0) = x, 
    \dot{\eta}\in 
    L^1\left([0,\varepsilon^{-1}]\right)
    \right\rbrace. \label{eq:optimalcontrolformula}
\end{align}
Without loss of generality we can assume $(x,t) = (0,1)$. Accordingly, we denote
\begin{equation*}
    \mathcal{A} =\big\lbrace \eta\in \mathrm{AC}([0,\varepsilon^{-1}]), \eta(0) = 0\big\rbrace
\end{equation*}
where $\mathrm{AC}([a,b])$ denotes the set of absolutely continuous functions from $[a,b]$ to $\mathbb{R}$. For $\eta\in \mathcal{A}$ we define the action functional
\begin{equation}\label{eq:A^eps}
    A^\varepsilon[\eta] = \varepsilon\int_0^{\varepsilon^{-1}} L\left(\eta(s), -\dot{\eta}(s)\right)\;ds + u_0\big(\varepsilon \eta(\varepsilon^{-1})\big). 
\end{equation}
If $\eta$ is a minimizer to \eqref{eq:optimalcontrolformula}, then it satisfies the Euler--Lagrange equation 
\begin{equation*}
    \begin{cases}
    \begin{aligned}
    \ddot{\eta}(s) &= -V'(\eta(s)) & & \text{in}\; \big(0,\varepsilon^{-1}\big),\\
    \eta(0) &=0. & &
    \end{aligned}
    \end{cases}
\end{equation*}
Due to the conservation of Hamilton's flow, there exists $r\in [r_{\textrm{min}},+\infty)$ where $r_\textrm{min} = \min_{\mathbb{R}}V$ such that
\begin{equation}\label{eq:rate-conservation-energy}
    H\left(\eta(s),\dot{\eta}(s)\right) = \frac{|\dot{\eta}(s)|^2}{2} + V(\eta(s)) = r \qquad\text{for all}\; s\in (0,\varepsilon^{-1}).
\end{equation}
Therefore, we can focus on the curves satisfying
\begin{equation}\label{eq:ode}
    \begin{cases}
    \begin{aligned}
        |\dot{\eta}(s)| &= \sqrt{2(r-V(\eta(s)))}, & & s\in (0,\varepsilon^{-1}),\\
        \eta(0) &= 0. & &
    \end{aligned}
    \end{cases}
\end{equation}
for any fixed energy $r\in [r_{\text{min}},+\infty)$. Benefiting from this, we can define
\begin{equation*}
    \mathcal{A}_r = \left\lbrace \eta\in \mathcal{A}\;\text{solving}\;\eqref{eq:ode}\;\text{with}\; H(\eta(s),\dot{\eta}(s)) = r\;\text{in}\;(0,\varepsilon^{-1})\right\rbrace.
\end{equation*}
Then 
\begin{equation}\label{eq:u-eps-01}
    u^\varepsilon(0,1) = \inf_{\eta\in \mathcal{A}} A^\varepsilon[\eta] = \inf_{r} 
    \left\lbrace 
    \inf_{\eta \in \mathcal{A}_r} A^\varepsilon[\eta] 
    \right\rbrace.
\end{equation}
To each $r \geq 0$, there exists a unique $p_r \geq p_0$ such that $\overline{H}(p_r) = r$, where $p_0 = \mathcal{M}(\mathbf{U}^{1/2})$ is defined as in \eqref{eq:defn-p0}, and
\begin{equation}\label{eq:defnOfpr}
    p_r = \mathcal{M}\left((2(r-V))^{1/2})\right) = \int_{\mathbb{T}^2} \sqrt{2(r+\mathbf{U}(x))}\;d\x = \mathcal{M}(\F_r). 
\end{equation}
where $\F_r = \sqrt{2}(r+\mathbf{U})^{1/2}$. By Lemma \ref{lem:basic-properties-H-bar} we have $\overline{H}(p_r) = r$. In what follows, we use 
 $\eta_r$ (or $\eta_{p_r}$) to indicate the path in $\mathcal{A}_r$ that corresponds to the energy $r$ (or momentum $p_r$), as long as  no ambiguity is caused.

\begin{lemma}\label{lem:TrickLemma} Assume \ref{itm:A0} and $V(x)\neq 0$ for all $x\in \R$. If $r\geq 0$ and $\eta_r$ is a solution to \eqref{eq:ode} with $\eta_r \geq 0$, then 
\begin{equation}\label{eq:clear-action-formula}
    A^\varepsilon[\eta_r] = -r + \left(\varepsilon \int_0^{\eta_r(\varepsilon^{-1})}
	 \sqrt{2(r-V(x))}\;dx \right) 
	 + u_0
	\left (\varepsilon
		\eta_r(\varepsilon^{-1}))
	\right).
\end{equation}
\end{lemma}
\begin{proof} If $r\geq 0$, then the solutions of \eqref{eq:ode} can only be one of the following possibilities:
\begin{equation}\label{eq:ode-positive}
\begin{cases}
\begin{aligned}
    \dot{\eta}(s) &= +\sqrt{2(r-V(\eta(s)))},  \qquad s\in (0,\infty),\\
    \eta(0) &= 0, 
\end{aligned}
\end{cases}
\end{equation}
or 
\begin{equation}\label{eq:ode-negative}
\begin{cases}
\begin{aligned}
    \dot{\eta}(s) &= -\sqrt{2(r-V(\eta(s)))},   \qquad s\in (0,\infty),\\
    \eta(0) &= 0,
\end{aligned}
\end{cases}
\end{equation} 
Indeed, since $r-V(x) > 0$ for every $r\geq 0$ and $x\in \mathbb{R}$, solutions to \eqref{eq:ode-positive} and \eqref{eq:ode-negative} exist for all time $s\in (0,\infty)$. Let us focus on \eqref{eq:ode-positive} as the other case is similar. We observe that
\begin{align}\label{eq:elliptic-integral-trick}
    \varepsilon^{-1} = \int_0^{\varepsilon^{-1}}\frac{\dot{\eta}(s)}{\dot{\eta}(s)}\;ds
	  &= \int_0^{\eta(\varepsilon^{-1})} \frac{dx}{\sqrt{2(r-V(x))}} 
	  \qquad \Longrightarrow\qquad 
        \varepsilon\int_0^{\eta(\varepsilon^{-1})} \frac{dx}{\sqrt{2(r-V(x))}} = 1.
\end{align}
Using the conservation of energy \eqref{eq:rate-conservation-energy}, we have
\begin{align}
    A^\varepsilon[\eta] 
        &= \left(r - 2\varepsilon \int_0^{\eta(\varepsilon^{-1})} 	
		\frac{V(x)}{\sqrt{2(r-V(x))}}dx
		+ u_0
		(\varepsilon
		\eta(\varepsilon^{-1}))
		\right) \label{eq:formula-1}\\
	&= r + 
	\left(\varepsilon \int_0^{\eta(\varepsilon^{-1})}
	 \sqrt{2(r-V(x))}\;dx \right) 
	 - 2r\left(
	 \varepsilon
	 \int_0^{\eta(\varepsilon^{-1})}\frac{dx}{\sqrt{2(r-V(x))}}
	 \right) 
	 + u_0
	\left (\varepsilon
		\eta(\varepsilon^{-1}))
	\right) . \nonumber
\end{align}
Using \eqref{eq:elliptic-integral-trick} we obtain the conclusion \eqref{eq:clear-action-formula}.
\end{proof}

Without loss of generality, we can always assume that \eqref{eq:ode-positive} has a smaller action value in \eqref{eq:A^eps} than
\eqref{eq:ode-negative}, since the arguments for these two are similar. Moreover, we can show that the minimization \eqref{eq:u-eps-01} cannot happen when $|r|$ is too large:

\begin{lemma}\label{prop:rlarge} There exists $r_0 = r_0(x)>0$ depends only on $\mathrm{Lip}(u_0)$ and $\Vert \mathbf{U}\Vert_{L^\infty}$ such that 
\begin{equation*}
    \inf_{r\geq r_0} A^\varepsilon[\eta] \geq u^\varepsilon(0, 1).
\end{equation*}
\end{lemma}

\begin{proof} Let $r>0$ and $\eta \in \mathcal{A}_r$ be a minimizer with energy $r$. Let $\widehat{u}(y,s) = u_0(y) + \widehat{C}s$ for $(y,s)\in \R\times (0,\infty)$ and $\widehat{C}$ large enough, depending on $\mathrm{Lip}(u_0)$ such that $\widehat{u}$ is a supersolution to \eqref{eq:intro:Ceps}. We have
\begin{align*}
    u^\varepsilon(0,1) \leq \widehat{u}(0,1) = u_0(0) + \widehat{C}
        &\leq u_0(\varepsilon \eta(\varepsilon^{-1})) + C|\varepsilon\eta(\varepsilon^{-1})| + \widehat{C}.
\end{align*}
Here $C$ is the Lipschitz constant of $u_0$. As $\eta\in \mathcal{A}_r$ implies $|\dot{\eta}|(s) = \sqrt{2(r-V(\eta_r(s)))}$, we have 
\begin{equation*}
    |\varepsilon
    \eta(\varepsilon^{-1})| \leq \sqrt{2}(|r| + \max |V|)^{1/2},
\end{equation*}
which implies that
\begin{equation*}
    u^\varepsilon(0,1) \leq u_0(\varepsilon\eta(\varepsilon^{-1})) + C\sqrt{2}(|r|+\max |V|)^{1/2}  + \widehat{C}. 
\end{equation*}
Therefore, if we choose
\begin{equation*}
    r_0 = C\sqrt{2}(|r|+\max |V|)^{1/2}  + \widehat{C},
\end{equation*}
then from \eqref{eq:formula-1} with $-V\geq 0$ and $\eta\geq 0$ we have $A^\varepsilon[\eta] \geq r + u_0(\varepsilon \eta(\varepsilon^{-1})) \geq u^\varepsilon(0,1)$.
\end{proof}

\begin{proposition}\label{thm:lower-bound2D} Assume \ref{itm:A1} or \ref{itm:A2} and $\sigma_\xi = 1$. There exists $C>0$ independent of $\varepsilon$ such that
\begin{numcases}
    {u^\varepsilon(0,1) - u(0,1)  \geq}
    -C\varepsilon^{\frac{\gamma}{\gamma+1}}
        &\text{if}\; $\gamma \in (0,1)$, 
        \nonumber\\
    -C\varepsilon^{1/2}
        &\text{if}\;$\gamma \in [1,2]$,
        \nonumber\\
    -C\varepsilon 
        &\text{if}\;$\gamma > 2$
        \label{eq:2DSobolev}.
\end{numcases}
\end{proposition}

\begin{proof} Let $r\in [r_{\text{min}},\infty)$ and $\eta\in \mathcal{A}_r$. For $|p| \geq |p_0|$, let $v_p$ be an exact sublinear corrector to the cell problem $H(x,p+Dv_p(x)) = \overline{H}(p)$ in $\mathbb{R}$ with $v_p(0) = 0$. We have 
    \begin{equation*}
        H\left(\eta(s), p + Dv_p(\eta(s))\right) = \overline{H}(p) \qquad\text{for all}\;s\in (0,\infty).
    \end{equation*}
    Therefore,  for $|p|\geq |p_0|$ we have \color{black}
    \begin{align*}
         \varepsilon \int_0^{\varepsilon^{-1}} 
        \Big( L(\eta(s), -\dot{\eta}(s))  + \overline{H}(p) \Big)\;ds     \geq
        -\varepsilon\int_0^{\varepsilon^{-1}}
            \dot{\eta}(s)
            \Big(p + Dv_p(\eta(s))\Big) \;ds 
            = 
        p \left(-\varepsilon \eta(\varepsilon^{-1})\right)- \varepsilon v_p\left(\eta(\varepsilon^{-1})\right).
    \end{align*}
    Hence,  for $|p|\geq |p_0|$ we have \color{black}
    \begin{align*}
       A^\varepsilon[\eta] 
            &= \varepsilon\int_0^{\varepsilon^{-1}}  L(\eta(s), -\dot{\eta}(s)) \;ds + u_0\left(\varepsilon\eta(\varepsilon^{-1})\right) \\
            & \geq 
            \left[ p \left(-\varepsilon \eta(\varepsilon^{-1})\right) - \overline{H} (p)\right] + u_0\left(\varepsilon\eta(\varepsilon^{-1})\right)  + \varepsilon v_p\left(\eta(\varepsilon^{-1})\right)\\
            & \geq \left[ p \left(-\varepsilon \eta(\varepsilon^{-1})\right) - \overline{H} (p)\right] + u_0\left(\varepsilon\eta(\varepsilon^{-1})\right)  + \inf_{|p|\geq |p_0|}\varepsilon v_p\left(\eta(\varepsilon^{-1})\right).
    \end{align*}
    As a consequence, we have 
    \begin{equation}\label{eq:deduction1}
        A^\varepsilon[\eta] - \inf_{|p|\geq |p_0|} \varepsilon v_p\left(\eta(\varepsilon^{-1})\right) 
            \geq   
        \left[p \left(-\varepsilon \eta(\varepsilon^{-1})\right) - \overline{H} (p)\right] + u_0\left(\varepsilon\eta(\varepsilon^{-1})\right) \quad \text{for any}\; |p|\geq |p_0|. 
    \end{equation}
    
    We deduce that, by changing the sign of $p$, we obtain 
    \begin{equation*}
        A^\varepsilon[\eta] - \inf_{|p|\geq |p_0|} \varepsilon v_p\left(\eta(\varepsilon^{-1})\right) 
            \geq   
        \left[|p| \cdot \big|\varepsilon \eta(\varepsilon^{-1})\big| - \overline{H} (p)\right] + u_0\left(\varepsilon\eta(\varepsilon^{-1})\right)\quad \text{for}\;|p|\geq |p_0|. 
    \end{equation*}
    In particular, at $|p| = |p_0|$, thanks to $\overline{H}(p_0) = \overline{H}(-p_0) = 0$, we have 
    \begin{equation*}
    \begin{aligned}
        A^\varepsilon[\eta] - \inf_{|p|\geq |p_0|} \varepsilon v_p\left(\eta(\varepsilon^{-1})\right) 
            &\geq   
        \left[|p_0| \cdot \big|\varepsilon \eta(\varepsilon^{-1})\big| \right] + u_0\left(\varepsilon\eta(\varepsilon^{-1})\right) \\
            &\geq 
        \left[|p| \cdot \big|\varepsilon \eta(\varepsilon^{-1})\big| \right] + u_0\left(\varepsilon\eta(\varepsilon^{-1})\right) \qquad\text{for any}\; p \in [-p_0,p_0] \\
            &= 
        \left[|p| \cdot \big|\varepsilon \eta(\varepsilon^{-1})\big| - \overline{H}(p)\right] + u_0\left(\varepsilon\eta(\varepsilon^{-1})\right) \qquad\text{for any}\; p \in [-p_0,p_0]
    \end{aligned}
    \end{equation*}
    thanks to $\overline{H}(p) = 0$ for $p\in [-p_0,p_0]$ from Lemma \ref{lem:basic-properties-H-bar}. We deduce that 
    \begin{equation}\label{eq:deduction2}
        A^\varepsilon[\eta] - \inf_{|p|\geq |p_0|} \varepsilon v_p\left(\eta(\varepsilon^{-1})\right) 
            \geq 
        \Big[ p \cdot \left(-\varepsilon \eta(\varepsilon^{-1})\right) - \overline{H}(p) \Big] + u_0\left(\varepsilon\eta(\varepsilon^{-1})\right) \quad\text{for any}\; p \in [-p_0,p_0].
    \end{equation}
    From \eqref{eq:deduction1} and \eqref{eq:deduction2}, we deduce that
    \begin{align}
        A^\varepsilon[\eta] - \inf_{|p|\geq |p_0|} \varepsilon v_p\left(\eta(\varepsilon^{-1})\right) 
            &\geq 
        \sup_{p\in \R}\Big[ p \cdot \left(-\varepsilon \eta(\varepsilon^{-1})\right) - \overline{H}(p) \Big] + u_0\left(\varepsilon\eta(\varepsilon^{-1})\right) \nonumber\\
            &= \overline{L}
        \left(
            -\varepsilon \eta(\varepsilon^{-1})
        \right) + u_0\big(\varepsilon\eta(\varepsilon^{-1})\big) \nonumber \\
            &\geq  \inf_{y\in \mathbb{R}} \Big( \overline{L}\left(-y\right) + u_0(y) \Big) = u(0,1) \label{eq:HopfLax}
    \end{align}
    where 
    \eqref{eq:HopfLax} comes from Hopf-Lax formula. In other words, we have
     \begin{align}\label{eq:bound-Aeps-tou01}
        A^\varepsilon[\eta]  \geq u(0,1) + \inf_{|p|\geq |p_0|} \varepsilon v_p\left(\eta(\varepsilon^{-1})\right). 
    \end{align}
    We can reduce $p_0\leq |p|\leq C_0$ as the minimization cannot happen when $|p|$ large (Proposition \ref{prop:rlarge}). Then $|\overline{H}(p)| \leq C_0$. Using Corollary \ref{coro:growth-vp-underA1} we have for $|p_0|\leq |p|\leq C$ that
    \begin{equation}\label{eq:lower2Drate}
        |\varepsilon v_p(\eta(\varepsilon^{-1}))|
        =
        |\varepsilon \eta(\varepsilon^{-1})|\cdot \left|
            \frac{v_p(\eta(\varepsilon^{-1}))}{\eta(\varepsilon^{-1})}
        \right|
        \leq 
    \begin{cases}
        C\varepsilon^{\frac{\gamma}{\gamma+1}}
            &\text{if}\;\gamma \in (0,1),\\     
        C\varepsilon^{\frac{1}{2}} 
            &\text{if}\;\gamma \in [1,2],\\     
        C\varepsilon 
            &\text{if}\;\gamma > 2, 
    \end{cases}
    \end{equation}
    where we use $|\varepsilon\eta(\varepsilon^{-1})| \leq C$ for $C$ depends on $C_0, \Vert \mathbf{U}\Vert_{L^\infty(\mathbb{T}^2)}$. From \eqref{eq:u-eps-01}, \eqref{eq:bound-Aeps-tou01} and \eqref{eq:lower2Drate} we obtain the conclusion. 
\end{proof}

\begin{proposition}\label{prop:upperbound2Dingredients} Assume \ref{itm:A1} or \ref{itm:A2} and $\sigma_\xi = 1$. There exists $C>0$ independent of $\varepsilon$ such that 
\begin{numcases}
    {u^\varepsilon(0,1) - u(0,1) \leq}
    \textstyle
    C\varepsilon^{\frac{\gamma-2}{3\gamma-2}}  
        &\text{if}\; $\gamma > 2$
        \label{eq:upper2Dbound}\\
        \textstyle
    \frac{C}{|\log(\varepsilon)|}
        &\text{if}\; $\gamma = 2$. 
        \nonumber
\end{numcases}
\end{proposition}

\begin{proof}[Proof] 
As $\gamma\geq 2$, we have $\overline{H}$ is $\C^1(\R)$ and
$\overline{H}'(p) = 0$ for $|p|\leq p_0$. In what follows, we provide a solution for \ref{itm:A1} and explain how \ref{itm:A2} follows in the same manner in the last step.

\smallskip
\paragraph{\underline{\textbf{Step 1. Positive energies}}} If $\eta_r \in \mathcal{A}_r$, then for $r\geq 0$ $\eta_r$ can only solve either \eqref{eq:ode-positive} or \eqref{eq:ode-negative}. Let us focus on $\eta_r \geq 0$, as the other case is similar. For $r\geq 0$ we have $r = \overline{H}(p_r)$ where $p_r$ is defined in \eqref{eq:defnOfpr}. Additionally, by \eqref{eq:representation-derivative-formula}, we have:
\begin{equation*}
    \lim_{s\to \infty}\frac{\eta_r(s)}{s} =
    \begin{cases}
        \overline{H}'(p_r) = \mathcal{M}\left(1/\F_r\right)^{-1}
            &\qquad\text{if}\;r > 0,    \\
        \overline{H}'(p_0) = 0 
            &\qquad\text{if}\;r = 0,
    \end{cases}
\end{equation*}
where $\F_r(\x) = (2(r+\mathbf{U}(\x)))^{1/2}$ for $\x \in \T^n$ as in \eqref{eq:defn-F-mu} (see $\K_r$ as in Lemma \ref{prop:propertiesV0}).  By Proposition \ref{coro:average-A1} \color{black} with $t = \varepsilon^{-1}$ we have
\begin{equation}\label{eq:rate2DinUpperBound}
    \left|
        \varepsilon \eta_r(\varepsilon^{-1}) - \overline{H}'(p_r)
    \right| 
    \leq 
    \begin{cases}
        C\varepsilon ^{\frac{\gamma-2}{3\gamma-2}} 
            &\text{if}\;\gamma > 2, \\
        \frac{C}{|\log(\varepsilon)|}
            &\text{if}\;\gamma  = 2.
    \end{cases}
\end{equation}
Here $C$ also depends on $t$ in the general case of $(x,t)$ instead of $(0,1)$, and also on $r_0$, due to the restriction $|r|\leq r_0$ due to Proposition \ref{prop:rlarge}. In view of \eqref{eq:clear-action-formula}, we define for for $r\geq 0$ the following quantities:
\begin{align*}
    A^{+}_r: &= -r + \overline{H}'(p_r) p_r + u_0\left(\overline{H}'(p_r)\right) \\
    A^{-}_r: &= -r + \overline{H}'(-p_r)(-p_r) + u_0\left(\overline{H}'(-p_r)\right) 
\end{align*}
For $r\geq 0$ and $\eta_r\in \mathcal{A}_r$ with $\eta_r\geq 0$, we compare $A^\varepsilon[\eta_r]$ and $A^+_r$ using:
\begin{align*}
    A^\varepsilon[\eta_r] - A^{+}_r
        &= \varepsilon \eta_r(\varepsilon^{-1}) \left(\frac{1}{\eta_r(\varepsilon^{-1})}\int_0^{\eta_r(\varepsilon^{-1})} (2(r+\mathbf{U}(\xi \x)))^{1/2}\;dx  - p_r\right)  \\
        &\qquad \qquad + \left( \varepsilon \eta_r(\varepsilon^{-1})  - \overline{H}'(p_r)\right) p_r + u_0\left(\varepsilon \eta_r(\varepsilon^{-1})\right) - u_0\left(\overline{H}'(p_r)\right).
\end{align*}
Therefore
\begin{align*}
    \left|A^\varepsilon[\eta_r] - A^{+}_r\right| 
        &\leq |\varepsilon\eta_r(\varepsilon^{-1})|\cdot \left|\frac{1}{\eta_r(\varepsilon^{-1})}\int_0^{\eta_r(\varepsilon^{-1})}\F_r(\xi x)\;dx - \mathcal{M}(\F_r) \right|  \\
        & \qquad\qquad \qquad \qquad\qquad\qquad + \left|\varepsilon \eta_r(\varepsilon^{-1})  - \overline{H}'(p_r)\right| \cdot \left(|p_r| + C\right),
\end{align*}
where $C = \mathrm{Lip}(u_0)$. Using Corollary \ref{coro:growth-vp-underA1} for the first term and \eqref{eq:rate2DinUpperBound} for the second term, we have:
\begin{itemize}
    \item If $\gamma > 2$, then, as $0\leq r\leq r_0$ we have
    \begin{align}\label{eq:gamma-2-range-bigger-2}
        \left|A^\varepsilon[\eta_r] - A^{+}_r\right|  
        &\leq 
            \varepsilon\eta_r(\varepsilon^{-1}) \cdot \frac{C}{|\eta_r(\varepsilon^{-1})|} 
            + 
            C\varepsilon^{\frac{\gamma-2}{3\gamma-2}}
        \leq 
            C\varepsilon + C\varepsilon^{\frac{\gamma-2}{3\gamma-2}} \leq C\varepsilon^{\frac{\gamma-2}{3\gamma-2}}. 
    \end{align}
    
    \item If $\gamma = 2$, then similarly to the previous case, we have
    \begin{align}\label{eq:gamma-2-range-equal-2}
        \left|A^\varepsilon[\eta_r] - A^{+}_r\right|  
        &\leq \varepsilon\eta_r(\varepsilon^{-1}) \cdot \frac{C}{|\eta_r(\varepsilon^{-1})|^{1/2}}
        +
        \frac{C}{|\log(\varepsilon)|} \leq C\varepsilon^{1/2} + \frac{C}{|\log(\varepsilon)|} \leq \frac{C}{|\log(\varepsilon)|}.
    \end{align}
\end{itemize}
Similarly, if $r\geq 0$ and $\eta_r \in \mathcal{A}_r$ with  $\eta_r\leq 0$\color{black}, we compare $A^\varepsilon[\eta_r]$ with $A_r^-$ and we can obtain the same rate as in \eqref{eq:gamma-2-range-bigger-2} and \eqref{eq:gamma-2-range-equal-2}.  \smallskip

\paragraph{\underline{\textbf{Step 2. Negative energies}}}

If $r<0$, from the equations \eqref{eq:ode-positive} and \eqref{eq:ode-negative} we have $\eta_0^-(s) \leq \eta_r(s) \leq \eta_0^+(s)$ for $s\in (0,\varepsilon^{-1})$, where $\eta_0^\pm$ are solutions to \eqref{eq:ode-positive}, \eqref{eq:ode-negative} with $r=0$, respectively. In particular, we deduce that 
\begin{equation*}
    0 \leq 
    |\varepsilon \eta_r(\varepsilon^{-1})|
    \leq 
    \max 
    \left\lbrace
    |\varepsilon \eta^-_0(\varepsilon^{-1}) - \overline{H}'(-p_0)|,
    |\varepsilon \eta^+_0(\varepsilon^{-1}) - 
    \overline{H}'(+p_0)|
    \right\rbrace.
\end{equation*}
We recall that $\overline{H}'(\pm p_0) = 0$. From \eqref{eq:A^eps}, \eqref{eq:rate-conservation-energy} and $L(x,v)\geq 0$ for all $(x,v)$ due to \ref{itm:A0}, we have 
\begin{equation*}
    \inf_{r\leq 0} A^\varepsilon[\eta_r] 
    \geq 
    u_0(\varepsilon \eta_r(\varepsilon^{-1})) 
    \geq 
    u_0(0) - C|\varepsilon\eta_0^-(\varepsilon^{-1})|
\end{equation*}
and thanks to \eqref{eq:clear-action-formula} we also have
\begin{align*}
     \inf_{r\leq 0} A^\varepsilon[\eta_r] 
        &\leq A^\varepsilon[\eta_0^+] = \varepsilon \int_0^{\eta_0^+(\varepsilon^{-1})} \sqrt{-2V(x)}\;dx + u_0(\varepsilon\eta_0^+(\varepsilon^{-1})) 
        \leq u_0(0) + C|\varepsilon \eta^+_0(\varepsilon^{-1})| .
\end{align*}
We conclude that 
\begin{equation}\label{eq:gamma-2-range-A}
    \left|
    \inf_{r\leq 0} A^\varepsilon[\eta_r] - u_0(0)
    \right| 
    \leq C|\varepsilon\eta_0(\varepsilon^{-1})|
    \leq \begin{cases}
            C \varepsilon^{\frac{\gamma-2}{3\gamma-2}}    & \gamma > 2, \\
            \frac{C}{|\log(\varepsilon)|} & \gamma = 2. 
        \end{cases}
    . 
\end{equation}

\paragraph{\underline{\textbf{Step 3. Combining the results}}} From Steps 1 and 2, we have
\begin{equation*}
    \left|
    \min
    \left\lbrace 
        \inf_{r\geq 0}A^\varepsilon[\eta_r],  
        \inf_{r\leq 0}A^\varepsilon[\eta_r]
    \right\rbrace 
    - 
    \min 
    \left\lbrace  
        \inf_{r\geq 0} A^{\pm}_r,\inf_{r\leq 0} u_0(0)
    \right\rbrace 
    \right| \leq 
        \begin{cases}
            C \varepsilon^{\frac{\gamma-2}{3\gamma-2}}    & \gamma > 2, \\
            \frac{C}{|\log(\varepsilon)|} & \gamma = 2. 
        \end{cases}
\end{equation*} 
Since $u^\varepsilon(0,1) = \inf_r A^\varepsilon[\eta_r]$ and $u^\varepsilon\to u$ as $\varepsilon\to 0^+$, we obtain 
\begin{equation*}
    \left|
    u^\varepsilon(0,1)
    - 
   u(0,1)
    \right| \leq 
        \begin{cases}
            C \varepsilon^{\frac{\gamma-2}{3\gamma-2}}    & \gamma > 2, \\
            \frac{C}{|\log(\varepsilon)|} & \gamma = 2,
        \end{cases}
\end{equation*}
together with a representation $u(0,1) = \min 
    \left\lbrace  
        \inf_{r>0} A^{\pm}_r,\inf_{r\leq 0} A^{\pm}_r
    \right\rbrace$.
\smallskip

\paragraph{\underline{\textbf{Step 4. Adaptation to the second prototype \ref{itm:A2}}}} Using the same strategy as in Subsection \ref{SubSection:rate2D}, we observe that for positive energy, the proof carries through similarly. The only difference occurs at $r\leq 0$, particularly $r=0$. Let us consider a path $\eta_0$ with a general point $(x,t)\in \R\times (0,\infty)$, we have

\begin{equation*}
    \begin{cases}
    \begin{aligned}
        |\dot{\eta}_0(s)| &= \sqrt{-2V(\eta_0(s))}, \qquad s\in (0,\varepsilon^{-1}t),\\
         \eta_0(0) &= \varepsilon^{-1}x.
    \end{aligned}
    \end{cases}
\end{equation*}

Under \ref{itm:A2}, we see that $V(x) = 0$ if and only if $x=0$, instead of $V(x)<0$ for all $x\in \R$ under \ref{itm:A1}. We note that since $\max_{\R} V = 0$, we also have
\begin{equation*}
    V'(0) = 0
\end{equation*}
and thus $x\mapsto \sqrt{-2V(x)}$ is Lipschitz at $x=0$ (see \cite{tu_2018_rate_asymptotic} for example). Depending on the initial position $\eta_0(0)$, the path $\eta_0$ can behave in the following ways:
\begin{itemize}
    \item $\eta_0(0) \equiv  0$, then it is easy to see that either $\varepsilon\eta_0 \equiv 0$, or $\varepsilon\eta_0(\varepsilon^{-1})$ behaves like \eqref{eq:gamma-2-range-A}. 
    \item $\eta_0(0) \neq 0$, then without loss of generality we consider $\eta_0(0) > 0$. There are only to solutions that satisfies either
    \begin{equation*}
            \lim_{s\to \infty} \eta_0(s) = +\infty \qquad\text{or}\qquad \lim_{s\to \infty} \eta_0(s) = 0.
        \end{equation*}
        In the latter case, the time it takes for $\eta_0$ to reach $0$ is infinite, since $V'(0) = 0$. In both cases, we have $0<\eta_0(s) \leq \eta_0^+(s)$ for all $s>0$, where $\dot{\eta}_0^+(s) = \sqrt{-2V(\eta_0^+(s))}$.
        We note that this solution $\eta_0^+$ satisfies \eqref{eq:gamma-2-range-A}, thus a bound for $\eta_0$ is obtain. 
\end{itemize}
In all cases, we obtain a similar estimate for $r=0$ as in \eqref{eq:gamma-2-range-A}, thus the proof proceeds similarly.
\end{proof}

\begin{proof}[Proof of Theorem \ref{thm:rate-C1-H-bar-prototype}] The proof of Theorem \ref{thm:rate-C1-H-bar-prototype} follows from Propositions \ref{thm:lower-bound2D} and \ref{prop:upperbound2Dingredients}. 
\end{proof}

\begin{corollary}\label{coro:RelaxationOfSigmaXi2} Assume \ref{itm:A1} and any $\xi \in \R^2$ with $\sigma_\xi < 2$, we have  $-C\varepsilon \leq u^\varepsilon(0,1) - u(0,1) \leq C\varepsilon^{\frac{\gamma-2}{3\gamma-2}}$.
\end{corollary}
\begin{proof} Due to \eqref{eq:LowerRateA1GammaGeneralizedXi2} in Corollary \ref{coro:rate-v0}, we can relax the condition $\sigma_\xi=1$ into $\sigma_\xi < 2$ in \eqref{eq:correctorBig2}, \eqref{eq:polyrateforgammaneq2}, \eqref{eq:2DSobolev} and \eqref{eq:upper2Dbound}. Put them together in the proof of Theorem \ref{thm:rate-C1-H-bar-prototype} we deduce the conclusion.
\end{proof}


\subsection{Proof of Theorem \ref{thm:n-frequency-lower-bound}}
We sketch the proof of Theorem \ref{thm:n-frequency-lower-bound} here, as it is similar to that of Theorem \ref{thm:rate-C1-H-bar-prototype}.

\begin{proof}[Proof of Theorem \ref{thm:n-frequency-lower-bound}] For the lower bound, we proceed similarly to Proposition \ref{thm:lower-bound2D} to obtain that 
\begin{align}\label{eq:bound-Aeps-tou01-new}
    A^\varepsilon[\eta_r]  \geq u(0,1) + \inf_{|p|\geq |p_0|} \varepsilon v_p\big(\eta(\varepsilon^{-1})\big). 
\end{align}

\paragraph{\underline{\textbf{Case 1}}} If \ref{itm:P1} holds then by Corollary \ref{prop:rate-corrector-sublinear-badly-approximable} with $|p_0|\leq |p|\leq C$, there exists $C=C\left(u_0, \Vert \mathbf{U}^{1/2}\Vert_{H^s(\T^n)}\right)$ such that
    \begin{equation*}
        \big|\varepsilon v_p\big(\eta(\varepsilon^{-1})\big)\big| 
        = \big|\varepsilon \eta(\varepsilon^{-1})\big|
    \cdot 
    \left| 
        \frac{v_p(\eta(\varepsilon^{-1}))}{\eta(\varepsilon^{-1})} 
    \right|
    \leq  \big|\varepsilon \eta(\varepsilon^{-1})\big|
    \cdot 
        \frac{C}{|\eta(\varepsilon^{-1})|} \leq C\varepsilon.
    \end{equation*}
    \smallskip
    
\paragraph{\underline{\textbf{Case 2}}} If \ref{itm:P2} holds then similarly
    \begin{align}
        \big|\varepsilon v_p\big(\eta(\varepsilon^{-1})\big)\big|
    &\leq  
        \varepsilon |\eta(\varepsilon^{-1})| \cdot 
    \frac{C|\log(\eta(\varepsilon^{-1}))|^{3(n-1)}}{|\eta(\varepsilon^{-1} )|^\theta} 
    = C\varepsilon |\eta(\varepsilon^{-1})|^{1-\theta} |\log(\eta(\varepsilon^{-1}))|^{3(n-1)} \label{eq:similarP2A}
    \end{align}
    where $\theta = \frac{\alpha}{\alpha+n-1}$ from Corollary \ref{prop:rate-corrector-sublinear-badly-approximable}. For $\varepsilon \in (0,1)$ and $|p_0|\leq |p|\leq C$, using the equation for $\eta$ we have $|\varepsilon\eta(\varepsilon
    ^{-1})| \leq C_0 := \sqrt{2}(|r_0| + \max |V|)^{1/2}$, where $r_0$ is the constant from Proposition \ref{prop:rlarge}. There are two cases:
    \begin{itemize} 
        \item[$\circ$] If $|\eta(\varepsilon^{-1})|\leq 1$ then from  we have
        \begin{align*}
             \big|\varepsilon v_p\big(\eta(\varepsilon^{-1})\big)\big| \leq
             C\varepsilon |\eta(\varepsilon^{-1})|^{1-\theta} |\log(\eta(\varepsilon^{-1}))|^{3(n-1)} \leq C\varepsilon.
        \end{align*}
        \item[$\circ$] If $|\eta(\varepsilon^{-1})|> 1$ then as $|\eta(\varepsilon^{-1})|\leq C_0\varepsilon^{-1}$, we have
        \begin{align*}
            |\log(\eta(\varepsilon^{-1}))| \leq \log(C_0) + |\log(\varepsilon)| \leq C|\log(\varepsilon)|
        \end{align*}
        if $\varepsilon$ is small enough. Therefore from \eqref{eq:similarP2A} we have
        \begin{align*}
            \big|\varepsilon v_p\big(\eta(\varepsilon^{-1})\big)\big| 
            &\leq C\left(|\varepsilon \eta(\varepsilon^{-1})|^{1-\theta} \right)\left(\varepsilon^\theta |\log(\varepsilon)|^{3(n-1)}\right) \leq C \varepsilon^\theta |\log(\varepsilon)|^{3(n-1)}.
        \end{align*}
    \end{itemize}
    From the two cases we obtain 
    \begin{equation*}
        \big|\varepsilon v_p\big(\eta(\varepsilon^{-1})\big)\big| 
            \leq 
        C \varepsilon^{\frac{\alpha}{\alpha+n-1}}|\log(\varepsilon)|^{3(n-1)}.
    \end{equation*}
    Here the constant $C$ depends on $\mathbf{U}, \xi, r_0$ and $\alpha$. \smallskip

\paragraph{\underline{\textbf{Case 3}}} If \ref{itm:P4New3} 
 or \ref{itm:P3} holds, then the result follow in the same manner, thanks to Corollary \ref{prop:rate-corrector-sublinear-badly-approximable}.
\smallskip

From \eqref{eq:bound-Aeps-tou01-new} and the corresponding estimates on $\varepsilon v_p(\eta(\varepsilon^{-1}))$ outlined in the steps above, we obtain the conclusions \eqref{eq:casesP1}, \eqref{eq:casesP2}, \eqref{eq:casesP4New3}, and \eqref{eq:casesP3} of Theorem \ref{thm:n-frequency-lower-bound}. For the upper bound, we proceed similarly to Proposition \ref{prop:upperbound2Dingredients}, omitting the details for brevity.
\end{proof}


\section{Application to Quantitative Ergodic Estimates and Discussions}
\label{sec:rateND}

\subsection{Improvement on the convergence rate of critical characteristics}

If $\overline{H}$ is $\C^{1,\beta}$ (even only one-sided) at $|p| = |p_0|$, the rate of convergence for the critical characteristic $\eta_0$ 
can be further improved, as a generalization of Corollary \ref{coro:average-A1}. Recall that Corollary \ref{coro:average-A1} claims that, under \ref{itm:A1} or \ref{itm:A2} and $\sigma_\xi = 1$, then
\begin{equation}\label{eq:rateoo}
\left|\frac{\eta_0(t)}{t}\right| \leq 
    C\left(\frac{1}{|t|}\right)^{\frac{\gamma-2}{3\gamma-2}}, \qquad \gamma > 2, 
\end{equation}
which can be improved as follows.

\begin{corollary}\label{coro:improve-rate0} Assume \ref{itm:A1} or \ref{itm:A2} with $\gamma>2$. Let $\eta_0$ be a characteristic with respect to $p = p_0$, then 
\begin{align*}
    \left|\frac{\eta_0(t)}{t}\right| \leq \left(\frac{1}{|t|}\right)^{\tau} 
        \qquad\text{where}\; \tau = \frac{(\gamma-2)(3\gamma-2)}{(\gamma-2)(3\gamma-2)+4\gamma^2} .
\end{align*}
\end{corollary}

\begin{proof}From \eqref{eq:estimate-p-tilde-p} we have
\begin{align}\label{eq:improveA}
    \left|\frac{\eta_0(t)}{t}\right|\cdot |p - p_0|\leq \overline{H}(p) + \frac{C}{|t|}. 
\end{align}
From \eqref{eq:HolderHBar} in Lemma \ref{prop:gamma=2-log-est}, we have $\overline{H}\in \C^{1,\beta}$ from one-sided at $p_0$, with $\beta = \frac{1}{2} - \frac{1}{\gamma}$ if $\gamma>2$. We have $\overline{H}(p) \leq C|p-p_0|^{1+\beta}$. Using \eqref{eq:bound-2-sides-DH-bar-gamma-not-2} we have
\begin{align*}
    \overline{H}(p) 
        &\leq \overline{H}'(p)\cdot (p-p_0) \leq C\overline{H}(p)^\beta |p-p_0| \leq C|p-p_0|^{\beta(1+\beta)}\cdot |p-p_0|.
\end{align*}
Together with \eqref{eq:improveA} we deduce that 
\begin{align*}
    \left|\frac{\eta_0(t)}{t}\right| \leq  C|p-p_0|^{\beta(1+\beta)} + \frac{C}{|t|\cdot|p-p_0|}. 
\end{align*}
Let $p = p_0 + \omega(t)$ where $\omega(t)$ is to be chosen (note that this implies $p>p_0$), we obtain
\begin{equation*}
    \left|\frac{\eta_0(t)}{t}\right|  \leq C\omega(t)^{\beta(\beta+1)} + \frac{C}{|t|\omega(t)}. 
\end{equation*}
The best choice of $\omega(t)$ is $ \omega(t)^{\beta(\beta+1) + 1} = |t|^{-1}$, thus $\omega(t) = |t|^{-\left(\frac{1}{\beta(\beta+1)+1}\right)}$, and consequently
\begin{equation*}
    \left|\frac{\eta_0(t)}{t}\right| \leq C\left(\frac{1}{|t|}\right)^{\tau} \qquad\text{where}\; \tau =  \frac{\beta(\beta+1)}{\beta(\beta+1)+1}.  
\end{equation*}
By evaluating $\beta = \frac{1}{2}-\frac{1}{\gamma}$, we establish the assertion.
\end{proof}

\begin{remark} Notice that when $\gamma>2$, $\tau \in \left(0,\frac{3}{7}\right)$, which is an improvement over the exponent in \eqref{eq:rateoo}. Such an improvement is not yet available for  $\gamma < 2$ because the associated $\overline{H}$ is not differentiable at $p_0$. In general, we can proposed the following open question. 
\end{remark}
\begin{question} Assume \ref{itm:A1} or \ref{itm:A2}. If $\eta_0$ is the characteristic at $p = p_0$, what is the optimal rate of convergence of 
\begin{equation*}
    \left| \frac{\eta_0(t)}{t}- \overline{H}'_+(p_0)\right|? 
\end{equation*}
\end{question}

\subsection{Application to quantitative ergodic estimates}

As a direct generalization of Proposition \ref{prop:ap:rate-Mean-QR} and \ref{prop:n-frequency-rate}, we look into  the rate of convergence to the mean value for unbounded quasi-periodic functions. In fact, it is not clear whether such a convergence is possible, not to mention the rate. 
 
\begin{question} Assume $f\in \mathrm{AP}(\R)\;\text{such that}\;f(x)>0 \;\text{for all}\;x\in \R$ and $\inf_{x\in\R}f(x)=0$.
\begin{itemize}
    \item[$\mathrm{(i)}$] When does the mean value theorem hold, i.e., 
    \begin{equation}\label{eq:QuestionErgodic}
        \lim_{T\to \infty} \frac{1}{T}\int_0^T \frac{dx}{f(x)} = \mathcal{M}\left(\frac 1{f}\right)? 
    \end{equation}         
    \item[$\mathrm{(ii)}$] If the mean value exists, can we determine a convergence rate?
\end{itemize}
\end{question}
 
To the best of our knowledge, no results related to this question have ever been made in the literature. Now we propose partial answer 
to this question in the quasi-periodic setting (Proposition \ref{prop:ASingleCaseOfApplicationErgodic}), by using the characteristics and non-resonant frequencies. To keep the presentation simple, we illustrate the idea using prototype \ref{itm:A1} and $f(x) = \mathbf{U}(\xi x)^{1/2}$. As for \ref{itm:A2},  we can modify the integral into 
\begin{equation*}
    \lim_{T\to \infty} \frac{1}{T}\int_a^{a+T} \frac{dx}{f(x)}
\end{equation*}
for some $a>0$, to avoid  $f(0) =0$. The argument is similar to \ref{itm:A1}.

\begin{proof}[Proof of Proposition \ref{prop:ASingleCaseOfApplicationErgodic}] Let $\eta_0$ be the characteristic with respect to $p=p_0$ such that $\eta_0(0) = 0$. From \eqref{eq:elliptic-integral-trick} with $r=0$ we have
\begin{equation}\label{eq:trick}
    \frac{t}{\eta_0(t)} = \frac{1}{\eta_0(t)} \int_0^{\eta_0(t)} \frac{dx}{\mathbf{U}(\xi x)^{1/2}}
\end{equation}
If $\gamma \geq 2$ then $\mathcal{M}(f^{-1}) = +\infty$ (Lemma \ref{lem:BehaviorDerevativeEffectiveHbar}). From Corollary \ref{coro:improve-rate0} we have 
    \begin{equation}\label{eq:ErgodicGammaBiggerEqual2}
        \left|\frac{\eta_0(t)}{t} \right| \leq 
        \begin{cases}
        \begin{aligned}
            &C|t|^{-\tau}  && \text{if}\;\gamma > 2, \\
            &C|\log(t)|^{-1} && \text{if}\;\gamma = 2.
        \end{aligned}
        \end{cases}
        \qquad\Longrightarrow\qquad \left|\frac{t}{\eta_0(t)}\right| \geq 
        \begin{cases}
            \begin{aligned}
                 &C|t|^{\tau}  && \text{if}\;\gamma > 2, \\
                &C|\log(t)| && 
                 \text{if}\;\gamma = 2.
            \end{aligned}
        \end{cases}
    \end{equation}
    where $\tau$ is given by \eqref{eq:RateQuasiPeriodicA1-A}. Using \eqref{eq:ErgodicGammaBiggerEqual2} in \eqref{eq:trick} we obtain 
    \begin{align}
        &&\frac{1}{T}\int_0^T \frac{dx}{\mathbf{U}(\xi x)^{1/2}} &\geq CT^{\tau} 
        && \text{if}\;\gamma > 2, \; \text{where}\; \tau = \frac{(\gamma-2)(3\gamma-2)}{(\gamma-2)(3\gamma-2)+4\gamma^2} 
    \label{eq:RateQuasiPeriodicA1-A} \\
        &&\frac{1}{T}\int_0^T \frac{dx}{\mathbf{U}(\xi x)^{1/2}} &\geq C|\log(T)|
        && \text{if}\;\gamma = 2. \nonumber
    \end{align}
If $\gamma \in (0,2)$ then $\mathcal{M}(f^{-1})$ is finite, by Lemma \ref{lem:BehaviorDerevativeEffectiveHbar}. 
\begin{itemize}
    \item If $\gamma \in [1, 2)$ then from \eqref{eq:polyrateforgammaneq2} of Proposition \ref{coro:average-A1} we have 
    \begin{equation*}
        \left|\frac{\eta_0(t)}{t} - \overline{H}'_+(p_0)\right| \leq \frac{C}{|t|^\tau}
        \qquad\text{where}\qquad \tau= \frac{1}{2}\frac{(2-\gamma)}{(2+\gamma)}. 
    \end{equation*}
    Since $\overline{H}'_+(p_0) > 0$, we deduce that for $t$ large enough then
    \begin{equation}\label{eq:BoundednessInverse}
        \left|\frac{t}{\eta_0(t)}\right| \leq C. 
    \end{equation}
    We deduce that 
    \begin{equation}\label{eq:ErgodicGammaTrick}
         \left|\frac{t}{\eta_0(t)} - \left(\overline{H}'_+(p_0)\right)^{-1}\right| \leq C \left(\overline{H}'_+(p_0)\right)^{-1} \left|\frac{t}{\eta_0(t)}\right|\cdot \frac{1}{|t|^\tau} \leq \frac{C}{|t|^\tau}
    \end{equation}
    for $|t|$ large enough, due to \eqref{eq:BoundednessInverse}. 
    From Lemma \ref{prop:D1H-bar-derivative} we have 
    \begin{equation}\label{eq:DerivativeFormula}
        \left(\overline{H}'_+(p_0)\right)^{-1} = \int_{\T^2} \frac{d\x}{\mathbf{U}(\x)^{1/2}}. 
    \end{equation}
    Using \eqref{eq:DerivativeFormula} in \eqref{eq:ErgodicGammaTrick} and \eqref{eq:trick} we obtain the conclusion 
    \begin{equation*}
        \left|\frac{1}{T}\int_0^T \frac{dx}{\mathbf{U}(\xi x)^{1/2}} - \int_{\T^2} \frac{d\x}{\mathbf{U}(\x)^{1/2}}\right| \leq C\left(\frac{1}{T}\right)^{\tau} \quad \text{where}\;\tau = \frac{1}{2}\cdot\frac{2-\gamma}{2+\gamma}. 
    \end{equation*}

    \item If $\gamma \in (0,1)\backslash \{\frac{2}{3}\}$, then by using Lemma \ref{prop:gamma=2-log-est} part (iii) and \eqref{eq:correctorLess2} of Corollary \ref{coro:growth-vp-underA1} we obtain
    \begin{numcases}
        {
        \left|\frac{\eta_0(t)}{t} - \overline{H}'_+(p_0)\right| \leq \frac{C}{|t|^\tau}
        \qquad\text{where}\;\tau = 
        }
        \textstyle \frac{\gamma}{1+\gamma}\frac{(2-\gamma)}{(2+\gamma)} 
            &\text{if}\; $\gamma \in (\frac{2}{3}, 1)$ 
            \nonumber, \\ 
        \textstyle \frac{1}{2}\frac{\gamma}{1+\gamma}. 
            &\text{if}\; $\gamma \in (0, \frac{2}{3})$ 
            \nonumber.
    \end{numcases}
    Using \eqref{eq:ErgodicGammaTrick} and \eqref{eq:DerivativeFormula}, we obtain 
    \begin{numcases}
        {\left|\frac{1}{T}\int_0^T \frac{dx}{\mathbf{U}(\xi x)^{1/2}} - \int_{\T^2} \frac{d\x}{\mathbf{U}(\x)^{1/2}}\right| \leq C\left(\frac{1}{T}\right)^{\tau} \quad \text{where}\;\tau = }
        \textstyle \frac{\gamma}{1+\gamma}\frac{2-\gamma}{(2+\gamma)}, 
            &\;\text{if}\; $\gamma \in (\frac{2}{3},1)$ 
            \nonumber, \\ 
        \textstyle \frac{1}{2}\frac{\gamma}{1+\gamma}, 
            &\;\text{if}\; $\gamma \in (0, \frac{2}{3})$ .
            \nonumber
    \end{numcases}

    \item If $\gamma = \frac{2}{3}$ then $\frac{\gamma}{\gamma+1} = \frac{2}{5}$, and from Lemma \ref{prop:gamma=2-log-est} part (iii) and \eqref{eq:correctorLess2} of Corollary \ref{coro:growth-vp-underA1} we have 
    \begin{align*}
        \left|\frac{\eta_0(t)}{t} - \overline{H}'_+(p_0)\right|
            &\leq C
        \Big|
            \mu + \mu \log\left(\frac{1}{\mu}\right)
        \Big| 
            + 
        \frac{C}{|t|^{2/5}\cdot |p - p_0|} 
            \leq C 
        \mu |\log (\mu)| + \frac{C}{|t|^{2/5}\cdot |p - p_0|}
    \end{align*}
    where for simplicity we denote $\mu = \overline{H}(p)$ with $p>p_0$. Let $\delta > 0$ and $C_H$ be the Lipschitz constant of $\overline{H}$, we choose $p = p_0 + C_H^{-1}|t|^{-\delta}$ so that $\mu = \overline{H}(p) \leq C_H|p-p_0| \leq |t|^{-\delta}$. Since $\mu\mapsto \mu|\log(\mu)|$ is non-decreasing for $\mu \in (0,1)$, we have
    \begin{align*}
        \left|\frac{\eta_0(t)}{t} - \overline{H}'_+(p_0)\right| \leq C\delta |t|^{-\delta} \log(|t|) + \frac{C C_H}{|t|^{2/5 - \delta}} = \frac{C\delta |\log(t)|}{|t|^\delta} + \frac{C}{|t|^{2/5-\delta}} \leq \frac{C|\log(t)|}{|t|^{1/5}}
    \end{align*}
    by choosing $\delta = \frac{1}{5}$. Using \eqref{eq:ErgodicGammaTrick} and \eqref{eq:DerivativeFormula} again we obtain
    \begin{equation*}
        \left|\frac{1}{T}\int_0^T \frac{dx}{\mathbf{U}(\xi x)^{1/2}} - \int_{\T^2} \frac{d\x}{\mathbf{U}(\x)^{1/2}}\right| \leq C\frac{|\log(T)|}{T^{1/5}}. 
    \end{equation*}
\end{itemize}
Finally, when $\gamma>2$, by \eqref{eq:LowerRateA1GammaGeneralizedXi2} in Corollary \ref{coro:rate-v0} we can relax the condition $\sigma_\xi = 1$ into $\sigma_\xi < 2$, in the same way as in Corollary \ref{coro:RelaxationOfSigmaXi2}. The proof is complete.
\end{proof}

\begin{remark}\label{rmk:GeneralErgodic} For a general function $\mathbf{U}$ satisfying \ref{itm:A0} and $f(x) = \mathbf{U}(\xi x)$ such that $f(x)>0$ for all $x\in \R$, the following procedure can be concluded for a rate of convergence in \eqref{eq:QuestionErgodic}:
\begin{enumerate}
    \item Let $H(x,p) = \frac{1}{2}|p|^2 - f(x)$. We compute the effective Hamiltonian $\overline{H}$ and analyze the regularity of $\overline{H}$ at the critical points $|p| = |p_0| = \int_{\T^n} \sqrt{2\mathbf{U}(\x)}\;d\x$ (see Lemma \ref{lem:basic-properties-H-bar}). This involves computations on how some norms of $\mathbf{U}^{1/2}$ blows up near its zero minimums.

    \item Obtaining the growth rate of a sublinear corrector at  $|p| = |p_0|$. This step utilizes either Proposition \ref{prop:ap:rate-Mean-QR} or \ref{prop:n-frequency-rate},  assuming conditions on the non-resonant vector $\xi$ such as \ref{itm:P1}, \ref{itm:P2}, \ref{itm:P4New3}, or \ref{itm:P3}. 

    \item Obtaining a convergence rate of characteristics at $|p| = |p_0|$ using the regularity of $\overline{H}$ at these points.

    \item Combining everything, we utilize an equation similar to \eqref{eq:trick} to link the ergodic estimate with critical characteristics. Such an equation, like \eqref{eq:trick} is always available due to the simple structure of the Euler-Lagrange equation for characteristics in one dimension (see Lemma \ref{lem:TrickLemma}). 
\end{enumerate}
\end{remark}


\subsection{Outlook and Discussion} 
For \ref{itm:A1} or \ref{itm:A2} with $\gamma\in(0,2)$, the associated effective Hamiltonian $\overline H(p)$ is not differentiable at $\pm p_0$. As a consequence, the {\it sub-gradient set} 
\begin{equation*}
    D^+ \overline{H}(p_0) = [0, \overline{H}'_+(p_0)]
\end{equation*}
establishes a non-empty interval. This situation prevents us from obtaining an effective upper bound for $u^\varepsilon(0,1)-u(0,1)$ as described in the Proposition \ref{prop:upperbound2Dingredients}. 
This is because for any $v\in (0, \overline{H}'_+(p_0))$, the possible characteristic with an average velocity  $v$ must have negative energy. For those characteristics, the large time behavior could not be derived from the effective Hamiltonian. On the other side, non-smooth $\overline H$ is rather common, e.g. $\gamma=1$ for prototype \ref{itm:A1} or \ref{itm:A2}. So we have to study the dynamics of characteristics with negative energy.

\begin{question} If $\overline{H}\notin \C^1(\R)$, can we obtain a polynomial upper bound rate for $u^\varepsilon(0,1) - u(0,1)$?
\end{question}


\section*{Acknowledgment}
B. Hu acknowledges support from SIMONS Travel Support MPS-TSM-00007213. S. Tu acknowledges support from NSF grants DMS-1664424, DMS-1843320, and DMS-2204722. J. Zhang acknowledges support from the National Key R\&D Program of China (No. 2022YFA1007500) and the National Natural Science Foundation of China (No. 12231010).
Son Tu extends gratitude to Hung Tran for suggesting the problem and for insightful discussions, and to Olga Turanova and Russell Schwab for their encouragement. Son Tu and Jianlu Zhang thank Hung Tran, Hiroyoshi Mitake, and Wenjia Jing for their valuable comments on the paper. The authors also thank the anonymous reviewers for their feedback, which greatly improved the paper, especially their insightful suggestions for clarifying the section on related literature and highlighting the paper's contributions. 

\bibliography{refs.bib}{}
\bibliographystyle{alpha}
\end{document}